\documentclass{amsart}
\usepackage[margin=4cm]{geometry}
\usepackage[T1]{fontenc}
\usepackage[utf8]{inputenc}
\usepackage{amsmath,amssymb,amsthm,mathtools}
\usepackage{enumitem}
\usepackage{microtype}
\usepackage{xcolor}
\usepackage{hyperref}
\usepackage{mathtools}

\DeclarePairedDelimiter\floor{\lfloor}{\rfloor}
\hypersetup{
  colorlinks=true,
  linkcolor=blue,
  citecolor=blue,
  urlcolor=blue
}
\newtheorem{theorem}{Theorem}[section]
\newtheorem{proposition}[theorem]{Proposition}
\newtheorem{lemma}[theorem]{Lemma}
\newtheorem{corollary}[theorem]{Corollary}

\theoremstyle{definition}
\newtheorem{definition}[theorem]{Definition}
\newtheorem{example}[theorem]{Example}

\theoremstyle{remark}
\newtheorem{remark}[theorem]{Remark}
\newtheorem{question}[theorem]{Question}
\newtheorem{maintheorem}{Theorem}

\newtheorem{maincorollary}[maintheorem]{Corollary}

\newcommand{\gir}{\operatorname{girth}}

\newcommand{\supp}{\operatorname{supp}}

\newcommand{\Alt}{\operatorname{Alt}}
\allowdisplaybreaks
\title[Infinite Girth, Finite Quotients, and Nielsen Connectivity in Groups]{Infinite Girth, Finite Quotients, and Nielsen Connectivity in Groups}

\author[Mishra]{Pratyush Mishra}
\address{HUN-REN Alfréd Rényi Institute of Mathematics, 13-15 street, Reáltanoda, 1053 Budapest, Hungary}
\email{mishrapratyushkumar@gmail.com}
\thanks{P.M. was supported by the National Research, Development and Innovation Office (NKFIH) Highlight Grant No. K153681}
\date{\today}
\begin{document}
\begin{abstract}
We develop new methods for studying infinite girth in finitely generated residually finite groups, particularly those containing no non-abelian free subgroups and satisfying no non-trivial group law. Our main contribution is a probabilistic Nielsen-component lifting principle, based on product replacement graphs, which provides a criterion for transferring suitable generation and word-avoidance properties of finite quotients to the ambient group. 

As an application, we prove that every finitely generated group admitting infinitely many pairwise non-isomorphic non-abelian finite simple quotients has infinite girth, with extensions to quasisimple and, under additional hypotheses, almost-simple quotients. We further establish infinite girth for the B. H. Neumann group and the Bou-Rabee–Seward group, obtaining new results on the occurrence of non-abelian finite simple quotients and sections. By a different construction, we also prove that \(\operatorname{Alt}_{\mathrm{fin}}(\mathbb Z)\rtimes\mathbb Z\) has infinite girth. Finally, we establish a marked finite-quotient characterization of infinite girth for finitely generated residually finite groups and prove that, for linear groups over fields of characteristic zero, large marked girth can be realized in congruence quotients. We conclude with several open questions and remarks arising from our results.

\end{abstract}
\maketitle
\enlargethispage{4pt}

\section{Introduction}\label{sec:introduction}
The girth of a finitely generated group \(G\), denoted by \(\operatorname{girth}(G)\), is defined as the supremum of the girths of its Cayley graphs \(\operatorname{Cay}(G,S)\), where \(S\) ranges over all finite generating sets of \(G\). The problem of determining whether a group has infinite girth has attracted considerable attention, particularly in connection with the existence of free subgroups and group laws.

For groups admitting suitable dynamical or geometric actions, an infinite girth can often be established using ping-pong arguments. These methods exploit the underlying action to construct generating sets with no short relations. However, they are generally unavailable for groups containing no non-abelian free subgroups. In this setting, establishing infinite girth requires different techniques. Indeed, the relatively few known examples of groups with infinite girth that contain no non-abelian free subgroups and satisfy no non-trivial group law--including Thompson's group \(F\) \cite{Brin2010, AkhmedovSteinTaback2011}, the first Grigorchuk group \cite{bartholdi2015ordering}, and certain iterated wreath products \cite[Lemma 2.3]{akhmedov2003girth}--rely on methods adapted to their particular actions on intervals or rooted trees. This leaves a natural question: **to what extent can infinite girth be established through intrinsic algebraic properties, without relying on a specific geometric or dynamical action?** The problem is especially compelling for residually finite groups, where finite quotients provide a potential source of information about the girth of the original group.

In this paper, we develop new methods for studying infinite girth through finite quotients, with particular emphasis on residually finite groups and groups without non-abelian free subgroups. Our motivation comes in part from constructions that exhibits interesting subgroup growth and residual finiteness properties, including those of Pyber \cite{MR2031168} and Bou-Rabee and Seward \cite{BouRabeeSeward+2016+199+204}. These constructions are closely related to B. H. Neumann's classical construction of a continuous family of pairwise non-isomorphic two-generated groups \cite{neumann1937some}.

A central theme of the paper is the relationship between the girth of a finitely generated group and the structure of its finite quotients. We thank László Pyber for suggesting the following two questions, which are also of independent interest.

\begin{question}\label{pyberfinitesimplenonabelian}
If a finitely generated group $\Gamma$ admits infinitely many non-isomorphic non-abelian finite simple quotients, must $\gir(\Gamma)=\infty$?
\end{question}

\begin{question}\label{compositionfactor}
Does there exist a finitely generated infinite group $\Gamma$ and infinitely many distinct finite-index normal subgroups $N_i\trianglelefteq\Gamma$ such that the finite quotients $\Gamma/N_i$, collectively, have infinitely many pairwise non-isomorphic non-abelian composition factors?
\end{question}

The first question is not straightforward. There are two principal obstructions to transferring large girth from finite quotients to the ambient group. First, a generating set of large girth in a quotient need not lift to a generating set of the original group with large girth. Second, the absence of non-trivial group laws alone does not guarantee infinite girth: Olshanski's result, in the formulation given in \cite[Theorem 4.1]{akhmedov2003girth}, provides groups with finite girth that satisfy no non-trivial law. Thus, a successful lifting argument must simultaneously address the generation and avoidance of short relations.

Our first main result is a probabilistic lifting criterion based on the product replacement graph and its Nielsen components. Roughly speaking, the criterion identifies conditions on a sequence of finite quotients under which generating tuples with prescribed properties can be transferred to the ambient group. The crucial point is that the argument uses redundant generating tuples and the connectivity properties of their Nielsen components, rather than requiring a particular generating set of a quotient to lift directly. To our knowledge, this is the first result of such kind.

\begin{maintheorem}\label{theroemA}[Redundant Nielsen-component lifting]
Let $\Gamma$ be a finitely generated group, let $d=d(\Gamma)$ denote the minimal cardinality of a generating set of $\Gamma$, and fix
\[
k\geq \max\{d+1,3\}.
\]
Suppose that there is a sequence of epimorphisms
\[
\pi_i\colon\Gamma\twoheadrightarrow Q_i
\]
onto finite groups satisfying the following conditions:
\begin{enumerate}
\item[\rm (i)] Each $Q_i$ is two-generated and has spread at least $2$.

\item[\rm (ii)] If $x_i,y_i$ are chosen independently and uniformly from $Q_i$, then
\[
\mathbb P\bigl(\langle x_i,y_i\rangle=Q_i\bigr)
\longrightarrow 1.
\]

\item[\rm (iii)] For every non-trivial word $w\in F_k$, if $\mathbf q_i$ is chosen uniformly from $Q_i^k$, then
\[
\mathbb P\bigl(w(\mathbf q_i)=1\bigr)
\longrightarrow 0.
\]
\end{enumerate}
Then, for every $R\geq1$, there exists a generating $k$-tuple $\mathbf S_R$ of $\Gamma$ such that
\[
\gir(\Gamma,\mathbf S_R)>R.
\]
\end{maintheorem}

Theorem \ref{theroemA} provides a general mechanism for obtaining infinite girth from suitable families of finite quotients. We use it to answer Question \ref{pyberfinitesimplenonabelian} affirmatively. We first establish the result for alternating groups and then extend it to arbitrary non-abelian finite simple quotients. In fact, the argument also applies to finite quasisimple quotients. The extension to almost-simple quotients requires an additional hypothesis concerning Nielsen components and the probability that a random tuple generates the quotient.

We record these extensions in the following two corollaries. Here, $\gir^{=}_k(\Gamma)$ denotes the supremum of the girths over finite generating sets of $\Gamma$ of cardinality exactly $k$.

\begin{maincorollary}
Suppose that a finitely generated group $\Gamma$ admits infinitely many pairwise non-isomorphic finite quasisimple quotients. Then, for every $k\geq d(\Gamma)+1$,
\[
\gir^{=}_k(\Gamma)=\infty.
\]
In particular, $\gir(\Gamma)=\infty\).
\end{maincorollary}

\begin{maincorollary}
Let $\Gamma$ be finitely generated, with $d=d(\Gamma)$, and let $k\geq d+1$. Suppose that $\Gamma$ admits epimorphisms
\[
\pi_i\colon\Gamma\twoheadrightarrow Q_i
\]
onto finite almost-simple groups satisfying:
\begin{enumerate}
\item[\rm (i)] All redundant generating $k$-tuples of each $Q_i$ lie in a single Nielsen component.

\item[\rm (ii)] For some integer $r<k$,
\[
\liminf_{i\to\infty}
\mathbb P\bigl(\langle q_{i,1},\ldots,q_{i,r}\rangle=Q_i\bigr)>0,
\]
where the $q_{i,j}$ are independent uniform elements of $Q_i$.

\item[\rm (iii)] For every non-trivial word $w\in F_k$,
\[
\mathbb P_{\mathbf q\in Q_i^k}\bigl(w(\mathbf q)=1\bigr)
\longrightarrow0.
\]
\end{enumerate}
Then
\[
\gir^{=}_k(\Gamma)=\infty.
\]
\end{maincorollary}

Our second main direction concerns Question \ref{compositionfactor}. We show that the B. H. Neumann group $\Gamma_0=\langle\pi,\tau\rangle$, constructed in \cite[Section 1]{MR2031168}, provides an affirmative answer. More strongly, its finite quotients can be chosen to be pairwise non-isomorphic non-abelian finite simple groups. In addition, the group has infinite girth despite containing no non-abelian free subgroup.

\begin{maintheorem}\label{BH}
The B. H. Neumann group $\Gamma_0=\langle\pi,\tau\rangle$ contains infinitely many distinct finite-index normal subgroups
\[
N_i\trianglelefteq\Gamma_0
\]
such that:
\begin{enumerate}
\item The quotients $\Gamma_0/N_i$ are pairwise non-isomorphic non-abelian finite simple groups.

\item $\gir(\Gamma_0)=\infty$.

\item $\Gamma_0$ contains no non-abelian free subgroup.
\end{enumerate}
In particular, the finite quotients of $\Gamma_0$ have infinitely many pairwise non-isomorphic non-abelian composition factors.
\end{maintheorem}

We next consider the finitary alternating group
\[
\Alt_{\mathrm{fin}}(\mathbb{Z)}\rtimes\mathbb{Z}.
\]
This group provides another example of a finitely generated group of infinite girth that contains no non-abelian free subgroup and satisfies no non-trivial group law. Its infinite girth, however, requires a different argument from the Nielsen-component lifting method. The key ingredient is an explicit construction, for each $N$, of a two-element generating set $\{\sigma_N,t\}$ whose Cayley graph has girth exceeding $N$. Although the fact that this group satisfies no non-trivial law was already established in \cite[Proposition 2.4]{akhmedov2003girth}, its girth had not previously been determined. We establish the following result.

\begin{maincorollary}\label{pybergroup}
We have
\[
\gir\bigl(
\Alt_{\mathrm{fin}}(\mathbb{Z})\rtimes\mathbb{Z}
\bigr)=\infty.
\]
Moreover, the B. H. Neumann group $\Gamma_0$ admits an epimorphism onto $\Alt_{\mathrm{fin}}(\mathbb Z)\rtimes\mathbb{Z}$.
\end{maincorollary}

In fact, the quotient relation gives a broader consequence: every finitely generated group admitting an epimorphism onto $\Alt_{\mathrm{fin}}(\mathbb Z)\rtimes\mathbb{Z}$ has infinite girth. This follows from the quotient principle of Akhmedov \cite[Proposition 1.1]{akhmedov2005girth}. Notice that the conclusion is substantially stronger than the absence of non-trivial laws. Indeed, although $\Gamma_0$ has infinite girth by Theorem \ref{BH}, the existence of groups satisfying no non-trivial law but having finite girth shows that the latter property alone is insufficient. The explicit generating sets constructed in the proof are therefore essential.

We also apply Theorem \ref{theroemA} to the BouRabee--Seward group $B_f$. Besides establishing infinite girth, we obtain information about the finite simple sections occurring inside a finite-index subgroup.

\begin{maintheorem}
The BouRabee--Seward group $B_f$ contains a proper finite-index normal subgroup $H\trianglelefteq B_f$ satisfying:
\begin{enumerate}
\item $\gir(B_f)=\infty$;

\item There exist finite subgroups $1\neq N_i\trianglelefteq K_i<H$ such that
\[
K_i/N_i\cong\Alt(n_i),
\]
where the groups $\Alt(n_i)$ are pairwise non-isomorphic non-abelian finite simple groups.
\end{enumerate}
Moreover, each $\Alt(n_i)$ occurs as a proper subgroup of $H$.
\end{maintheorem}

Finally, motivated by the lifting results above, we investigate how the girth of a residually finite group is reflected in its finite quotients. We introduce a marked version of the finite-quotient criterion, in which the generating set is retained when passing to the quotient. This leads to the following characterization, which is an elementary consequence of residual finiteness and the characterization of infinite girth through asymptotic freeness due to Arzhantseva \cite{arzhantseva2014asymptotic} [Arz14, Definition 3.1 and Theorem 3.3].

\begin{maintheorem}\label{markedfinitequotient1}[Marked finite-quotient criterion]
Let $\Gamma$ be a finitely generated residually finite group. The following are equivalent:
\begin{enumerate}
\item $\gir(\Gamma)=\infty$.

\item For every $N\geq1$, there exist a finite generating set $S_N$ of $\Gamma$ and a finite quotient
\[
\pi_N\colon\Gamma\twoheadrightarrow Q_N
\]
such that
\[
\gir(Q_N,\pi_N(S_N))>N.
\]
\end{enumerate}
\end{maintheorem}

We conclude by strengthening this criterion for finitely generated linear groups over fields of characteristic zero. In this setting, the finite quotients realizing large marked girth can be taken to be congruence quotients. Thus, the girth information can be detected through suitable reductions modulo maximal ideals.

\begin{maintheorem}[Congruence realization of marked girth]
Let $\Gamma\leq\operatorname{GL}_d(K)$ be a finitely generated linear group over a field $K$ of characteristic zero. Suppose that $\gir(\Gamma)=\infty$. Then, for every $N\in\mathbb N$, there exist a finite generating set $S_N$ of $\Gamma$, a finitely generated integral domain $A\subset K$ with
\[
\Gamma\leq\operatorname{GL}_d(A),
\]
and a maximal ideal $\mathfrak m_N$ of $A$ such that the congruence quotient
\[
\psi_{\mathfrak m_N}(\Gamma)
\leq\operatorname{GL}_d(A/\mathfrak m_N)
\]
satisfies
\[
\gir\bigl(
\psi_{\mathfrak m_N}(\Gamma),
\psi_{\mathfrak m_N}(S_N)
\bigr)>N.
\]
\end{maintheorem}

Taken together, these results provide several complementary approaches to infinite girth. They show how probabilistic properties of finite quotients, Nielsen connectivity, explicit generating-set constructions, and congruence reductions can be used to establish large girth in groups for which traditional ping-pong methods are unavailable. In particular, they demonstrate that the finite quotient structure of a group can encode substantially more information about its Cayley graphs than might be apparent from the existence of large-girth quotients alone.

\section{An Affirmative Answer to Question \ref{pyberfinitesimplenonabelian}}\label{Nielsencomponent}
\noindent We say a group $\Gamma$ is \textit{residually finite} if for any non-identity element $g\in \Gamma$, there exists a finite group $F$ and a homomorphism $\phi: \Gamma \to F$ such that $\phi(g)\neq e$.

Even before investigating Question \ref{pyberfinitesimplenonabelian}, the first question that comes to mind is dropping the "non-abelian finite simple" assumption :
\begin{question}\label{pyberresidualfiniteness1}
    If a residually finite group $\Gamma$ possesses finite quotients with arbitrarily large girth, does it follow that $\Gamma$ itself has infinite girth?
\end{question}

The Question \ref{pyberresidualfiniteness1} as stated is easily seen to be false in the abelian setting. For example, $\mathbb{Z}^2$ reduction$\mod{n}$, $\mathbb{Z}^2\to\mathbb{Z}/n\mathbb{Z}$ is a counterexample, and $\gir(\mathbb{Z}^2)=4$. In Example \ref{counterexample}, we utilize $\mathbb{Z}^2$ to construct a separating family of finite quotients with increasing girth, despite the fact that $\gir(\mathbb{Z}^2) = 4$.
\begin{definition}
    A family of group homomorphisms $T_k:G\to H$ is a separating family if \[\bigcap_{k=1}^{\infty}\ker(T_k)=e.\]
\end{definition}

\noindent In fact, the following example produces a separating family consisting entirely of cyclic quotients of growing order in $\mathbb{Z}^2$.

\begin{example}\label{counterexample} 
We show that $\mathbb{Z}^2$ possesses a separating family of finite quotients $\{Q_n\}_{n\in\mathbb{N}}$ such that $\lim_{n \to \infty} \gir(Q_n) = \infty$, although $\gir(\mathbb{Z}^2) < \infty$.
    
Let $\{p_k\}_{k \in \mathbb{N}}$ be a sequence of primes such that $p_k > k^2$ for all $k$. For each $k$, define the homomorphism $\phi_k \colon \mathbb{Z}^2 \to \mathbb{Z}/p_k\mathbb{Z}$ by
    \begin{equation}\label{equationdefinition}
        \phi_k(m, n) = (m + kn) \pmod{p_k}.
    \end{equation}
    Each $\phi_k$ is surjective, as $\phi_k(1, 0) = 1$. To show that the family $\{\mathbb{Z}/p_k\mathbb{Z}\}_{k\in\mathbb{N}}$ is separating, suppose $(m, n) \in \bigcap_{k=1}^{\infty} \ker(\phi_k)$. Then $p_k$ must divide $m + kn$ for all $k \geq 1$. Since $p_k > k^2$, the asymptotic behavior of the linear function $g(k) = m + kn$ relative to the quadratic growth of $p_k$ implies that
    \[ \lim_{k \to \infty} \frac{|m + kn|}{p_k} = 0. \]
    Thus, there exists $K \in \mathbb{N}$ such that for all $k \geq K$, we have $|m + kn| < p_k$. Given that $p_k$ divides $m + kn$, it must be that $m + kn = 0$ for all $k \geq K$. This forces $m = n = 0$, proving that $\bigcap_{k=1}^{\infty} \ker(\phi_k) = \{(0,0)\}$.

    Regarding the girth, the quotients $Q_k = \mathbb{Z}/p_k\mathbb{Z}$ are cyclic groups of order $p_k$. Since they satisfy the law $x^{p_k} = 1$, they possess finite girth; specifically, $\gir(\mathbb{Z}/p_k\mathbb{Z}) = p_k$ (attained by the standard generating set). It follows that $\lim_{k \to \infty} \gir(Q_k) = \infty$.
    
    On the other hand, $\gir(\mathbb{Z}^2) = 4$. Since any generating set $S$ of $\mathbb{Z}^2$ consists of commuting elements, any two generators $s_1, s_2 \in S$ satisfy the relation $s_1 s_2 s_1^{-1} s_2^{-1} = 1$, which is a word of length 4. Thus, $\gir(\mathbb{Z}^2) \leq 4$, and equality is realized by the standard basis $S = \{(1,0), (0,1)\}$.
    
\end{example}

\begin{example}\label{nonabelianKleinbottle}
    In the non-abelian setting, the answer to Question \ref{pyberresidualfiniteness1} also remains negative. The fundamental group of the Klein bottle, 
    \[ \pi_1(\mathbb{K}) = \langle a, b \mid aba^{-1} = b^{-1} \rangle, \]
    provides a counterexample. $\pi_1(\mathbb{K})$ is torsion-free (as both $a$ and $b$ have infinite order) and virtually abelian, containing the subgroup $\langle a^2, b \rangle \cong \mathbb{Z}^2$ of index $2$. Note that $\pi_1(\mathbb{K})$ can be realized as the Baumslag--Solitar group $BS(1, -1)$, which is known to be residually finite \cite[Theorem C]{MR285589}. Since $\pi_1(\mathbb{K})$ is virtually abelian, it follows that $\operatorname{girth}(\pi_1(\mathbb{K})) < \infty$ \cite[Lemma 3.1]{Schleimer2000}. 

    However, one can easily construct a family of finite quotients of $\pi_1(\mathbb{K})$ with arbitrarily large girth. For each $n \in \mathbb{N}$, define the homomorphism
    \[ \phi_n : \pi_1(\mathbb{K}) \to \mathbb{Z}/n\mathbb{Z} = \langle c \mid c^n = 1 \rangle \]
    by setting $\phi_n(a) = c$ and $\phi_n(b) = 1$. Since $\phi_n(aba^{-1}) = c \cdot 1 \cdot c^{-1} = 1 = \phi_n(b^{-1})$, the map $\phi_n$ is well-defined; surjectivity follows immediately. Moreover, $\operatorname{girth}(\mathbb{Z}/n\mathbb{Z}, \{c\}) = n$, which implies that $\operatorname{girth}(\mathbb{Z}/n\mathbb{Z}) \to \infty$ as $n \to \infty$.
\end{example}

\begin{remark}
    The failure of the implication in Example \ref{counterexample} and \ref{nonabelianKleinbottle} arises because a generating set $S_k$ realizing large girth in a quotient $Q_k$ does not necessarily lift to a generating set of $\Gamma$ that preserves this girth.
\end{remark}

We now show that the group as in Question \ref{pyberfinitesimplenonabelian} do not satisfy a law.

\begin{proposition}\label{nolawforfinitesimple}
    Let $\Gamma$ be a finitely generated group. If $\Gamma$ admits infinitely many pairwise non-isomorphic, non-abelian finite simple quotients $\{Q_n\}_{n \in \mathbb{N}}$, then $\Gamma$ satisfies no non-trivial group law.
\end{proposition}

\begin{proof}
    Suppose, toward a contradiction, that $\Gamma$ satisfies a non-trivial group law. That is, there exists a non-trivial word $w(x_1, \dots, x_k) \in \mathbb{F}_k$ such that $w(g_1, \dots, g_k) = 1$ for all $g_1, \dots, g_k \in \Gamma$.
    
    Let $\pi_n \colon \Gamma \twoheadrightarrow Q_n$ be the surjective homomorphisms onto the finite simple quotients. Since $w$ is a law in $\Gamma$, it must necessarily descend to a law in each quotient $Q_n$; that is, $w(q_1, \dots, q_k) = 1$ for all $q_1, \dots, q_k \in Q_n$ and for all $n\geq 1$.
    
    However, a result of G. A. Jones \cite{Jones_1974} establishes that for any non-trivial word $w$, there exists a constant $N_w$ such that $w$ is not a law for any finite simple group $G$ with $|G| > N_w$. Since the $\{Q_n\}$ are pairwise non-isomorphic finite simple groups, their orders $|Q_n|$ must tend to infinity as $n \to \infty$. Consequently, the word $w$ can be a law for only finitely many of the quotients $Q_n$. This contradicts the assumption that $w$ is a law for all $Q_n$, completing the proof.
\end{proof}

\begin{remark}\label{nonvirtuallysolvableofGamme}
    Any group $\Gamma$ satisfying the hypotheses of Proposition \ref{nolawforfinitesimple} cannot be virtually solvable. Indeed, if $\Gamma$ contains a solvable subgroup $\Gamma'$ of finite index $d=[\Gamma:\Gamma']$, the law $w$ satisfied by $\Gamma'$ would induce a non-trivial law $w(x_1^{d!}, \dots, x_n^{d!}) = 1$ for the entire group $\Gamma$. This contradicts the fact that $\Gamma$ satisfies no non-trivial group law.
\end{remark} 

If we restrict our attention to the class of groups satisfying the Girth Alternative--which asserts that a finitely generated group in the class has infinite girth if and only if it is not virtually solvable--we obtain an affirmative answer to Question \ref{pyberfinitesimplenonabelian}. Specifically, by Remark \ref{nonvirtuallysolvableofGamme}, since $\Gamma$ is not virtually solvable, it follows that $\gir(\Gamma) = \infty$. This applies, for instance, to groups belonging to classes where the Girth Alternative is known to hold, including linear groups, word hyperbolic groups, one relator groups \cite{akhmedov2005girth}, mapping class groups \cite{nakamura2014girth}, HNN extensions \cite{MR5091911}, groups acting on CAT(0) cube complexes \cite{banerjee2026girth}, and most recently for the groups admitting topologically free extreme boundary actions \cite{amrutam2026boundary}. 

\subsection{Resolving Question \ref{pyberfinitesimplenonabelian} via product replacement graph}Now, we say more to Question \ref{pyberfinitesimplenonabelian} than we discussed earlier in this section, providing a general affirmative answer in Theorem \ref{thm:redundant-nielsen-lifting}. Our approach departs from the study of arbitrary lifts of generating sets and instead leverages the structure of the product replacement graph. We are grateful to Miklós Abért for suggesting that the properties of this graph could provide a powerful avenue to answer Question \ref{pyberfinitesimplenonabelian}. The use of the product replacement framework to solve deep problems in group theory has been highlighted by recent landmark results, most notably the breakthrough work of Kaluba, Kielak, and Nowak \cite{kaluba2021property}, in which they proved that $\operatorname{Aut}(F_n)$ has the Kazhdan property (T) for every $n \geq 6$.

Let $\mathbf{F}_k=F(a_1,\ldots,a_k)$ denote the free group of rank $k$, equipped
with the standard word length $|\cdot|$ with respect to
$\{a_1^{\pm1},\ldots,a_k^{\pm1}\}$.

If $\mathbf{g}=(g_1,\ldots,g_k)\in G^k$ is a $k$-tuple, define the evaluation
homomorphism
\[
  \operatorname{ev}_{\mathbf{g}}\colon \mathbf{F}_k\longrightarrow G,
  \qquad
  a_j\longmapsto g_j.
\]
The tuple $\mathbf{g}$ is a generating $k$-tuple of $G$ if
$\operatorname{ev}_{\mathbf{g}}$ is surjective.

\begin{definition}
Let $\mathbf{g}=(g_1,\ldots,g_k)$ be a generating $k$-tuple of a
group $G$. Its marked girth is
\[
  \gir(G,\mathbf{g})
  :=
  \min\bigl\{
       |w| :
       1\neq w\in \mathbf{F}_k,\ 
       \operatorname{ev}_{\mathbf{g}}(w)=1
     \bigr\},
\]
with the convention that the minimum is $\infty$ if
$\operatorname{ev}_{\mathbf{g}}$ is injective.
For a finitely generated group $G$, set
\[
  \gir^{=}_k(G)
  :=
  \sup\bigl\{
       \gir(G,\mathbf{g}) :
       \mathbf{g}\in G^k,\ 
       \langle g_1,\ldots,g_k\rangle=G
     \bigr\}.
\]
\end{definition}

Thus $\gir^{=}_k(G)$ is the supremum over ordered
generating $k$-tuples. For tuples with marked girth greater than $2$,
the entries are non-trivial and pairwise distinct, so the same tuple
also gives a generating set of cardinality exactly $k$.

Recall that a generating $k$-tuple
$\mathbf{g}=(g_1,\ldots,g_k)$ is called \emph{redundant} if one of
its entries can be omitted while the remaining entries still generate
the group. Equivalently, after applying Nielsen transformations,
a redundant tuple can be put in the form
\[
  (h_1,\ldots,h_{k-1},1).
\]

Two generating $k$-tuples are \emph{Nielsen equivalent} if they belong
to the same $\operatorname{Aut}(F_k)$-orbit. Equivalently, they can be
connected by a finite sequence of elementary Nielsen moves, such as
\[
  g_i\longmapsto g_i^{-1},
  \qquad
  g_i\longmapsto g_i g_j^{\pm1},
  \qquad
  g_i\longmapsto g_j^{\pm1}g_i
  \quad (i\neq j),
\]
together with permutations of the coordinates. Recently, in a paper under review \cite{mishra2024dynamics}, the author of this paper introduced the notion of Nielsen girth and Schreier girth using similar analogy to the product replacement graph and prove the existence of
groups with infinite girth but having finite Nielsen girth, which can be seen as another source of motivation for the result of this kind.

Let us recall that a \textit{word} $w$ is an element of a free group $F_d=\langle x_1,x_2,\dots, x_d\rangle$. For a given finite group $G$, the word $w$ gives rise to a word map $w=w_G:G^d\to G$ defined by substitution.

\begin{definition}
    For a finite group $G$, define 
    \begin{equation}
        \mathbb{P}(w=1)=\frac{|\ker(w_G)|}{|G|^d},
    \end{equation}
    the probability that $w(g_1,g_2,\dots,g_d)=1,$ where $g_i$'s are chosen randomly.
\end{definition}
It is easy to see that $w=1$ is a non-trivial law of $G$ if and only if $\mathbb{P}(w=1)=1$. When $w\neq 1$ and $G$ is a non-abelian finite simple group, if follows from \cite{Jones_1974} that $\mathbb{P}(w=1)<1$ provided $|G|$ is large enough. Further, we use the following notation to describe the probability that a random tuples of element satisfies a specific relation:
\begin{enumerate}
    \item $\mathbf{q}\in Q^k$: We choose $k$-elements $(q_1,q_2,\dots,q_k)$ from the finite group $Q$ uniformly at random.
    \item $\mathbb{P}_{\mathbf{q}\in Q^k} \bigl(w(\mathbf{q})=1\bigr)$: The probability that the random tuple $\mathbf{q}$ satisfies the word equation $w=1$.
\end{enumerate}

We shall use the following elementary lifting observation.

\begin{lemma}[Nielsen moves and quotients]
\label{lem:nielsen-quotient-lifting}
Let $\pi\colon G\twoheadrightarrow Q$ be an epimorphism, let
$\mathbf{g}\in G^k$ be a generating tuple, and let
$\varphi\in\operatorname{Aut}(F_k)$. Then:
\begin{enumerate}
\item the tuple $\varphi\cdot\mathbf{g}$ generates $G$;
\item
\[
   \pi(\varphi\cdot\mathbf{g})
   =
   \varphi\cdot\pi(\mathbf{g});
\]
\item
\[
   \operatorname{girth}(G,\varphi\cdot\mathbf{g})
   \geq
   \operatorname{girth}(Q,\varphi\cdot\pi(\mathbf{g})).
\]
\end{enumerate}
\end{lemma}

\begin{proof}
The first assertion follows because precomposition of a surjective map
$F_k\twoheadrightarrow G$ with an automorphism of $F_k$ remains
surjective. The second assertion follows from the functoriality of word
evaluation.

For the third assertion, suppose that
\[
  w(\varphi\cdot\mathbf{g})=1
\]
for some non-trivial reduced word $w\in F_k$. Applying $\pi$ gives
\[
  w(\varphi\cdot\pi(\mathbf{g}))=1.
\]
Consequently, every relation satisfied by
$\varphi\cdot\mathbf{g}$ is also satisfied by its image in $Q$, which
proves the asserted inequality.
\end{proof}

Recall that a group $Q$ has \emph{spread at least $2$} if, for every
pair of non-trivial elements $x_1,x_2\in Q$, there exists $y\in Q$
such that
\[
  \langle x_1,y\rangle
  =
  \langle x_2,y\rangle
  =
  Q.
\]
A result recorded as Proposition~2.5.12 of
\cite{MR1829489}, and attributed there to Evans
\cite{MR1188815}, states that if $Q$ is two-generated and has spread at
least $2$, then, for every $k\geq 3$, all redundant generating
$k$-tuples of $Q$ belong to the same component of the product
replacement graph. In particular, all such tuples are Nielsen
equivalent.

We now formulate the lifting criterion.

\begin{theorem}[Redundant Nielsen-component lifting]
\label{thm:redundant-nielsen-lifting}
Let $\Gamma$ be a finitely generated group, let $d=d(\Gamma)$ be the cardinality of the minimal generating set of $\Gamma$
and fix an integer
\[
  k\geq \max\{d+1,3\}.
\]
Suppose that there is a sequence of epimorphisms
\[
  \pi_i\colon \Gamma\twoheadrightarrow Q_i
\]
onto finite groups satisfying the following conditions:
\begin{enumerate}
\item[\rm (i)]
each $Q_i$ is two-generated and has spread at least $2$;

\item[\rm (ii)]
if $x_i,y_i$ are chosen independently and uniformly from $Q_i$, then
\[
  \mathbb{P}\bigl(\langle x_i,y_i\rangle=Q_i\bigr)
  \longrightarrow 1;
\]

\item[\rm (iii)]
for every non-trivial word $w\in F_k$, if
$\mathbf{q}_i$ is chosen uniformly from $Q_i^k$, then
\[
  \mathbb{P}\bigl(w(\mathbf{q}_i)=1\bigr)
  \longrightarrow 0.
\]
\end{enumerate}
Then
\[
  \operatorname{girth}^{=}_k(\Gamma)=\infty.
\]
More precisely, for every $R\geq 1$, there exists a generating
$k$-tuple $\mathbf{S}_R$ of $\Gamma$ such that
\[
  \operatorname{girth}(\Gamma,\mathbf{S}_R)>R.
\]
\end{theorem}

\begin{proof}
Choose a generating $d$-tuple
\[
  (x_1,\ldots,x_d)
\]
of $\Gamma$, and stabilize it to a generating $k$-tuple
\[
  \mathbf{S}
  :=
  (x_1,\ldots,x_d,
    \underbrace{1,\ldots,1}_{k-d\text{ entries}})
  \in\Gamma^k.
\]
Since $k\geq d+1$, the tuple $\mathbf{S}$ contains at least one trivial
entry. Therefore, for every $i$, the image
\[
  \pi_i(\mathbf{S})
\]
is a redundant generating $k$-tuple of $Q_i$.

Fix $R\geq 1$, and let
\[
  \mathcal{W}_{k,R}
  :=
  \{w\in F_k: w\neq 1,\ |w|\leq R\}.
\]
The set $\mathcal{W}_{k,R}$ is finite.

Choose a uniformly random tuple
\[
  \mathbf{q}_i=(q_{i,1},\ldots,q_{i,k})\in Q_i^k.
\]
Consider the two events
\[
  A_i
  :=
  \{\langle q_{i,1},q_{i,2}\rangle=Q_i\}
\]
and
\[
  B_i
  :=
  \{w(\mathbf{q}_i)\neq 1
       \text{ for every }w\in\mathcal{W}_{k,R}\}.
\]
By assumption~{\rm (ii)},
\[
  \mathbb{P}(A_i)\longrightarrow 1.
\]
Moreover, by the union bound and assumption~{\rm (iii)},
\[
\begin{aligned}
  \mathbb{P}(B_i^c)
  &=
  \mathbb{P}\bigl(
       \text{there exists }w\in\mathcal{W}_{k,R}
       \text{ such that }w(\mathbf{q}_i)=1
     \bigr)
\\
  &\leq
  \sum_{w\in\mathcal{W}_{k,R}}
  \mathbb{P}\bigl(w(\mathbf{q}_i)=1\bigr)
  \longrightarrow 0.
\end{aligned}
\]
Hence
\[
  \mathbb{P}(A_i\cap B_i)
  \geq
  1-\mathbb{P}(A_i^c)-\mathbb{P}(B_i^c)
  \longrightarrow 1.
\]
In particular, for all sufficiently large $i$, there exists
\[
  \mathbf{y}_i=(y_{i,1},\ldots,y_{i,k})\in Q_i^k
\]
such that
\[
  \langle y_{i,1},y_{i,2}\rangle=Q_i
\]
and
\[
  w(\mathbf{y}_i)\neq 1
  \qquad
  \text{for every }w\in\mathcal{W}_{k,R}.
\]
The first condition implies that $\mathbf{y}_i$ is a redundant
generating $k$-tuple, while the second condition gives
\[
  \operatorname{girth}(Q_i,\mathbf{y}_i)>R.
\]

Both $\pi_i(\mathbf{S})$ and $\mathbf{y}_i$ are redundant generating
$k$-tuples of $Q_i$. Since $Q_i$ is two-generated and has spread at
least $2$, Proposition~2.5.12 of \cite{MR1829489} implies that the two
tuples belong to the same component of the product replacement graph.
Thus there is a finite sequence of Nielsen moves, equivalently some
$\varphi_i\in\operatorname{Aut}(F_k)$, such that
\[
  \varphi_i\cdot\pi_i(\mathbf{S})
  =
  \mathbf{y}_i.
\]

Apply the same Nielsen transformations upstairs and put
\[
  \mathbf{S}_R:=\varphi_i\cdot\mathbf{S}.
\]
By Lemma~\ref{lem:nielsen-quotient-lifting}, the tuple
$\mathbf{S}_R$ generates $\Gamma$, and
\[
  \pi_i(\mathbf{S}_R)
  =
  \varphi_i\cdot\pi_i(\mathbf{S})
  =
  \mathbf{y}_i.
\]
The quotient-girth inequality now gives
\[
  \operatorname{girth}(\Gamma,\mathbf{S}_R)
  \geq
  \operatorname{girth}(Q_i,\mathbf{y}_i)
  >R.
\]
Since $R$ was arbitrary, it follows that
\[
  \operatorname{girth}^{=}_k(\Gamma)=\infty.
\]
\end{proof}

\begin{remark}
The proof of Theorem~\ref{thm:redundant-nielsen-lifting} is
existential. It uses the connectedness of the redundant component of
the product replacement graph, but it does not require a bound on its
diameter or on the mixing time of the product replacement random walk.
In particular, rapid mixing is not needed for the girth conclusion.
\end{remark}

\begin{remark}
For $R\geq 2$, the tuple $\mathbf{S}_R$ has $k$ distinct non-trivial
entries. Indeed, $S_{R,j}=1$ would give a relation of length $1$, while
$S_{R,r}=S_{R,s}$ for $r\neq s$ would give the non-trivial freely
reduced relation $a_ra_s^{-1}$ of length $2$. Thus the theorem also
applies when $\operatorname{girth}^{=}_k(\Gamma)$ is defined using
generating sets of cardinality exactly $k$, rather than ordered
generating tuples.
\end{remark}

\subsection{Application of Theorem \ref{thm:redundant-nielsen-lifting}: Alternating quotients}

\begin{corollary}[Alternating simple quotients]
\label{cor:alternating-quotients}
Let $\Gamma$ be a finitely generated group, and suppose that there are
epimorphisms
\[
  \pi_i\colon\Gamma\twoheadrightarrow \Alt(n_i),
\]
where $n_i\to\infty$. Then, for every
\[
  k\geq d(\Gamma)+1,
\]
one has
\[
  \operatorname{girth}^{=}_k(\Gamma)=\infty.
\]
In particular,
\[
  \operatorname{girth}(\Gamma)=\infty.
\]
\end{corollary}

\begin{proof}
Since $\Gamma$ has a non-abelian quotient, it is not cyclic, and hence
$d(\Gamma)\geq 2$. Thus $k\geq d(\Gamma)+1$ implies $k\geq 3$.

We verify the three hypotheses of
Theorem~\ref{thm:redundant-nielsen-lifting}.

First, the alternating group $\Alt(n)$ is a non-abelian finite simple group for $n\geq 5$.
Moreover, alternating groups are two-generated and have spread at
least $2$. The spread assertion also follows from the stronger theorem
of Breuer--Guralnick--Kantor that every non-abelian finite simple group
has uniform spread at least $2$; see
\cite[Corollary~1.3]{MR2422303}.

Second, a theorem of Liebeck and Shalev states that, as $Q$ ranges over
non-abelian finite simple groups with $|Q|\to\infty$, the probability
that two independent uniformly distributed elements generate $Q$
tends to $1$; see \cite{MR1338320}. In particular,
\[
  \mathbb{P}\bigl(
     \langle x_i,y_i\rangle=\Alt(n_i)
  \bigr)
  \longrightarrow 1
\]
for independent uniformly distributed $x_i,y_i\in \Alt(n_i)$.

Third, let $1\neq w\in \mathbf{F}_k$ be fixed. Larsen and Shalev proved that
there is a constant $\varepsilon(w)>0$ such that, for every
sufficiently large non-abelian finite simple group $Q$,
\[
  \mathbb{P}_{\mathbf{q}\in Q^k}
  \bigl(w(\mathbf{q})=1\bigr)
  \leq
  |Q|^{-\varepsilon(w)};
\]
see \cite[Theorem~1.1]{MR2879221}. Since
$|\Alt(n_i)|\to\infty$, it follows that
\[
  \mathbb{P}_{\mathbf{q}\in \Alt(n_i)^k}
  \bigl(w(\mathbf{q})=1\bigr)
  \longrightarrow 0.
\]

All the hypotheses of
Theorem~\ref{thm:redundant-nielsen-lifting} are therefore satisfied,
and the conclusion follows.
\end{proof}

\begin{remark}
The standard generating pair of $\operatorname{Alt}(n)$
\[
 \alpha_n = (1, 2, 3, \dots, n) \quad \text{and} \quad \beta = (1, 2, 3).
\]
should not be used as the desired large-girth pair, since
$\beta^3=1$ and hence the corresponding marked Cayley graph has girth
at most $3$. The proof of
Corollary~\ref{cor:alternating-quotients} instead uses random
redundant tuples of large girth and connects them, by Nielsen moves,
to the image of a stabilized generating tuple of $\Gamma$.
\end{remark}

\subsection{Application of Theorem \ref{thm:redundant-nielsen-lifting}: Arbitrary non-abelian finite simple quotients}

\begin{corollary}[Infinitely many finite simple quotients]
\label{cor:general-simple-quotients}
Let $\Gamma$ be a finitely generated group. Suppose that $\Gamma$
admits infinitely many pairwise non-isomorphic non-abelian finite
simple quotients. Then, for every
\[
  k\geq d(\Gamma)+1,
\]
one has
\[
  \operatorname{girth}^{=}_k(\Gamma)=\infty.
\]
Consequently,
\[
  \operatorname{girth}(\Gamma)=\infty.
\]
\end{corollary}

\begin{proof}
Choose pairwise non-isomorphic non-abelian finite simple quotients
$Q_i$ and epimorphisms
\[
  \pi_i\colon\Gamma\twoheadrightarrow Q_i.
\]
There are only finitely many finite groups of bounded order.
Consequently, after passing to a subsequence if necessary,
\[
  |Q_i|\longrightarrow\infty.
\]

Since $\Gamma$ maps onto a non-abelian group, $d(\Gamma)\geq 2$.
Therefore $k\geq d(\Gamma)+1$ implies $k\geq 3$.

By \cite[Corollary~1.3]{MR2422303}, every non-abelian
finite simple group has uniform spread at least $2$, and hence has
spread at least $2$. In particular, each $Q_i$ is two-generated and
satisfies hypothesis~{\rm (i)} of
Theorem~\ref{thm:redundant-nielsen-lifting}.

By \cite{MR1338320}, if $x_i,y_i$ are independent uniformly
distributed elements of $Q_i$, then
\[
  \mathbb{P}\bigl(\langle x_i,y_i\rangle=Q_i\bigr)
  \longrightarrow 1.
\]
This is hypothesis~{\rm (ii)}.

Finally, let $1\neq w\in F_k$. By
\cite[Theorem~1.1]{MR2879221}, there exists
$\varepsilon(w)>0$ such that, for every sufficiently large
non-abelian finite simple group $Q$,
\[
  \mathbb{P}_{\mathbf{q}\in Q^k}
  \bigl(w(\mathbf{q})=1\bigr)
  \leq
  |Q|^{-\varepsilon(w)}.
\]
Since $|Q_i|\to\infty$, we obtain
\[
  \mathbb{P}_{\mathbf{q}\in Q_i^k}
  \bigl(w(\mathbf{q})=1\bigr)
  \longrightarrow 0.
\]
Thus hypothesis~{\rm (iii)} is satisfied.

Applying Theorem~\ref{thm:redundant-nielsen-lifting} gives
\[
  \operatorname{girth}^{=}_k(\Gamma)=\infty.
\]
Since this holds for at least one finite value of $k$, it also follows
that
\[
  \operatorname{girth}(\Gamma)=\infty.
\]
\end{proof}

\begin{remark}
The proof only uses the following part of the product replacement
theory: all redundant generating $k$-tuples lie in one Nielsen
component. It does not require the full product replacement graph
$\Gamma_k(Q_i)$ to be connected.
\end{remark}

\subsection{Application of Theorem \ref{thm:redundant-nielsen-lifting}: Quasi-simple  and almost-simple quotients}
Let us recall that a finite group is \textit{quasi-simple} if $Q=[Q,Q]$ and $Q/Z(Q)$ is a non-abelian finite simple group.
\begin{corollary}[Quasisimple quotients]
\label{cor:quasisimple-quotients}
Let $\Gamma$ be a finitely generated group admitting infinitely many
pairwise non-isomorphic finite quasisimple quotients
\[
    \pi_i\colon\Gamma\twoheadrightarrow Q_i.
\]
Then, for every
\[
    k\geq d(\Gamma)+1,
\]
one has
\[
    \operatorname{girth}^{=}_k(\Gamma)=\infty.
\]
In particular,
\[
    \operatorname{girth}(\Gamma)=\infty.
\]
\end{corollary}

\begin{proof}
For every $i$, put
\[
    S_i:=Q_i/Z(Q_i).
\]
Since $Q_i$ is quasisimple, $S_i$ is a non-abelian finite simple
group. Composing $\pi_i$ with the canonical central quotient gives an
epimorphism
\[
    \overline{\pi}_i\colon
    \Gamma
      \xrightarrow{\pi_i}
    Q_i
      \twoheadrightarrow
    Q_i/Z(Q_i)=S_i.
\]

We claim that the set of isomorphism types among the $S_i$ is
infinite. Suppose, on the contrary, that only finitely many simple
groups occur among the $S_i$. Passing to an infinite subsequence, we
may assume that
\[
    S_i\cong S
\]
for one fixed non-abelian finite simple group $S$.

Let
\[
    \widetilde{S}\twoheadrightarrow S
\]
be the universal central extension of $S$. Every quasisimple group
$Q$ with
\[
    Q/Z(Q)\cong S
\]
is a quotient of $\widetilde{S}$ by a central subgroup. Since
$\widetilde{S}$ is finite, it has only finitely many subgroups and
hence only finitely many quotient groups up to isomorphism. Therefore
there are only finitely many quasisimple groups $Q$ satisfying
\[
    Q/Z(Q)\cong S.
\]
This contradicts the assumption that the groups $Q_i$ are pairwise
non-isomorphic.

Thus, after passing to a subsequence, the groups $S_i$ are pairwise
non-isomorphic non-abelian finite simple quotients of $\Gamma$.
By Corollary \ref{cor:general-simple-quotients}, we have
\[
    \operatorname{girth}^{=}_k(\Gamma)=\infty
\]
for every $k\geq d(\Gamma)+1$.
\end{proof}
 For almost simple groups, the reduction is not automatic and additional assumptions are needed. Let us recall that a group $Q$ is almost-simple if it contains a non-abelian simple group $S$ (also called the socle $S$ of $Q$), then $$S\leq Q\leq \operatorname{Aut}(S).$$

 \begin{theorem}[Almost-simple Nielsen lifting criterion]
\label{thm:almost-simple-nielsen-lifting}
Let $\Gamma$ be a finitely generated group with
\[
    d=d(\Gamma),
\]
and let
\[
    k\geq d+1.
\]
Suppose that there are epimorphisms
\[
    \pi_i\colon\Gamma\twoheadrightarrow Q_i
\]
onto finite almost-simple groups such that:
\begin{enumerate}
    \item[\rm (i)]
    all redundant generating $k$-tuples of $Q_i$ lie in a single
    Nielsen component;

    \item[\rm (ii)]
    there is an integer $r<k$ such that
    \[
        \liminf_{i\to\infty}
        \mathbb{P}\bigl(
          \langle q_{i,1},\ldots,q_{i,r}\rangle=Q_i
        \bigr)>0,
    \]
    where $q_{i,1},\ldots,q_{i,r}$ are independent uniformly
    distributed elements of $Q_i$;

    \item[\rm (iii)]
    for every non-trivial word $w\in F_k$,
    \[
        \mathbb{P}_{\mathbf q\in Q_i^k}
        \bigl(w(\mathbf q)=1\bigr)
        \longrightarrow0.
    \]
\end{enumerate}
Then
\[
    \operatorname{girth}^{=}_k(\Gamma)=\infty.
\]
\end{theorem}

\begin{proof}
Choose a generating $d$-tuple
\[
    (x_1,\ldots,x_d)
\]
of $\Gamma$, and stabilize it to
\[
    \mathbf S
      =
    (x_1,\ldots,x_d,1,\ldots,1)\in\Gamma^k.
\]
Then $\pi_i(\mathbf x)$ is a redundant generating $k$-tuple of
$Q_i$.

Fix $R\geq1$, and let
\[
    \mathcal W_{k,R}
      =
    \{1\neq w\in F_k:|w|\leq R\}.
\]
This is a finite set.

Choose a uniformly random tuple
\[
    \mathbf q_i=(q_{i,1},\ldots,q_{i,k})\in Q_i^k.
\]
Let
\[
    A_i
      =
    \{\langle q_{i,1},\ldots,q_{i,r}\rangle=Q_i\}
\]
and
\[
    B_i
      =
    \{w(\mathbf q_i)\neq1
      \text{ for every }w\in\mathcal W_{k,R}\}.
\]
By assumption~{\rm (ii)}, there exists $\delta>0$ such that
\[
    \mathbb P(A_i)\geq\delta
\]
for all sufficiently large $i$.

By assumption~{\rm (iii)} and the union bound,
\[
\begin{aligned}
    \mathbb P(B_i^c)
    &\leq
    \sum_{w\in\mathcal W_{k,R}}
    \mathbb P\bigl(w(\mathbf q_i)=1\bigr)
      \longrightarrow0.
\end{aligned}
\]
Consequently,
\[
    \mathbb P(A_i\cap B_i)
    \geq
    \mathbb P(A_i)-\mathbb P(B_i^c)>0
\]
for all sufficiently large $i$.

Therefore there exists a tuple
\[
    \mathbf y_i=(y_{i,1},\ldots,y_{i,k})
\]
such that
\[
    \langle y_{i,1},\ldots,y_{i,r}\rangle=Q_i
\]
and
\[
    \gir(Q_i,\mathbf y_i)>R.
\]
Since $r<k$, the tuple $\mathbf y_i$ is redundant.

By assumption~{\rm (i)}, the redundant tuples
$\pi_i(\mathbf S)$ and $\mathbf y_i$ are Nielsen equivalent.
Therefore there exists
\[
    \varphi_i\in\operatorname{Aut}(F_k)
\]
such that
\[
    \varphi_i\cdot\pi_i(\mathbf S)=\mathbf y_i.
\]
Set
\[
    \mathbf S_R:=\varphi_i\cdot\mathbf S.
\]
Then $\mathbf x_R$ generates $\Gamma$, and
\[
    \pi_i(\mathbf S_R)=\mathbf y_i.
\]
It follows that
\[
    \gir(\Gamma,\mathbf S_R)
    \geq
    \gir(Q_i,\mathbf y_i)>R.
\]
Since $R$ was arbitrary,
\[
    \gir^{=}_k(\Gamma)=\infty.
\]
\end{proof}

\begin{remark}
    Theorem \ref{thm:almost-simple-nielsen-lifting} is valid for almost-simple quotients, but its three hypotheses must be checked for the particular family. There are relevant results on random generation of almost-simple groups, but they do not, by themselves, prove the required single Nielsen-component assertion for every almost-simple group \cite{fulman2024probabilistic}.
A useful sufficient condition for the Nielsen-component hypothesis is that $Q_i$ be two-generated and have spread at least $2$. The theorem of Burness, Guralnick and Harper says that a finite group has spread at least $2$ precisely when every proper quotient is cyclic. This can apply to some almost-simple families, but not all of them \cite{burness2021spread}.
\end{remark}

\section{Application of Theorem \ref{thm:redundant-nielsen-lifting}: Girth of the B.H. Neumann group}\label{pyberneumann} 
Recall that \[\operatorname{Sym}_\mathrm{fin}(X)=\{\sigma\in \operatorname{Sym}(X),\hspace{0.1cm}\supp(\sigma)<\infty\}\]
\[\operatorname{Alt}_\mathrm{fin}(X)=\{\sigma\in \operatorname{Sym}(X), \hspace{0.1cm}\supp(\sigma)<\infty,\hspace{0.1cm}\sigma\hspace{0.1cm}\text{is an even permutation}\}\]

Let $s_n(\Gamma)$ denote the number of subgroups of index $n$ in a finitely generated residually finite group $\Gamma$. For a given function $f$, $\Gamma$ is said to be of strict growth type $f$ if there are $b,c>0$ such that $f(n)^b\leq s_n(\Gamma)\leq f(n)^c$ for all $n\geq n_0$. For many year, a question of central interest in the theory of subgroup growth was to investigate which functions actually appears as the type of growth function $s_n(\Gamma)$. In \cite{MR2031168}, Pyber settled this question by proving that there exists 4-generated group $\Gamma$ having strict growth type of $f$, where $f(n)=n^{g(n)}$, and $g(n)$ is a non-decreasing function satisfying $g(n)\leq \log{n}$ with $g(n)=o(n)$. 

The terminology and terms in this section are borrowed from \cite{MR2031168}, which we present here for completeness.

For each $i$, choose a finite set \[\Omega_i=\{w_i^1,w_i^2,\dots, w_i^{n_i}\}\] where $7\leq n_1\leq n_2\leq n_3\leq \dots$

Let $f$ be the function $f:\mathbb{N}\to\mathbb{N}$ that specifies how often each size occurs: \[\# \{i:n_i=n\}=f(n).\] We set \[\Omega=\bigsqcup_{i=1}^{\infty} \Omega_i.\]
Hence, $\Omega$ is a countable disjoint union of finite blocks. The group constructed in \cite{MR2031168} which acts on each block separately.

On each block $\Omega_i$, consider the alternating group \[
\operatorname{Alt}(\Omega_i)\cong A_{n_i}.\] Following the notations as defined in \cite{MR2031168}, define the unrestricted direct product, \[G=\prod_{i=1}^{\infty} \operatorname{Alt}(\Omega_i)\]

As the direct product is unrestricted, for any $g=(g_1, g_2, g_3, \dots)\in G$, there is no requirement that all but finitely many $g_i\in \operatorname{Alt}(\Omega_i)$ are non-trivial. Finite groups are profinite with the discrete topology and arbitary product of profinite group is profinite (by  Tychonoff's theorem), so $G$ is a profinite group. This is same as the product topology on $G$. As mentioned in \cite[Section 1]{MR2031168}, the restricted direct product \[D=\bigoplus_{i=1}^{\infty} \operatorname{Alt}(\Omega_i)\] has profinite completion $\widehat{D}$ isomorphic to $G$.

Then we choose the special set $L_i$, which is defined as follow: Choose a non-decreasing unbounded sequnce of positive integer $l_1\leq l_2\leq l_3\leq \dots$ such that $\frac{n_i}{l_i}\geq 2$. Define \[v_i=\floor*{\frac{n_i}{l_i}-1}l_i.\] Clearly, each $v_i$ is strictly less than $n_i$ and $v_i$ is a multiple of $l_i$. Now define, \[L_i=\{w_i^1,w_i^2,\dots, w_i^{v_i}\}\subseteq \Omega_i\] \[L=\bigcup_{i=1}^{\infty} L_i.\] Each $L_i$ is a large portion of $\Omega_i$ and $\Omega_i\backslash L_i$ is a non-empty "reserved part" of each block $\Omega_i$. Let $Q$ is a subgroup of $G$ that satisfies the following conditions\begin{enumerate}
    \item $Q$ is supported inside $L$: For each $q=(q_1,q_2,\dots)\in Q$ acts as a permutation $q_i$ in each factor set $\Omega_i$, that is $q_i\in \operatorname{Alt}(\Omega_i),$ \[ \supp (q)=\{x\in \Omega: q(x)\neq x\}=\bigsqcup_{i=1}^{\infty} \supp(q_i)\subseteq L\] More precise, for each $1\leq i\leq \infty$ and $q\in Q$, we have \[\supp q_i\subseteq L_i.\] That means, $q_i$ fixes the reserved block $\Omega_i\backslash L_i$ for each $i$, \[q_i(w_i^{j})=w_i^j\hspace{0.1cm}\text{for all}\hspace{0.1cm} v_i<j\leq n_i.\] Hence, the subgroup $Q$ is only allowed to permute within $L_i$ in each block $\Omega_i$ and acts trivial on the tail set $\{w_i^{v_{i+1}}, w_i^{v_{i+2}},\dots, w_i^{n_i}\}$
    \item $Q\cap D=1$: The subgroup $Q$ do not contain any non-trivial element of the dense discrete subgroup $D$.
    \item Distinguishes equal-sized blocks: If $n_i= n_j$, for any $i\neq j$, the the corresponding subgroup $Q_i$ and $Q_j$ acting on sets $\Omega_i$ and $\Omega_j$ must be different, that is $Q_i\neq Q_j$ and $Q_i\leq \operatorname{Sym}(\Omega_i)\cong S_{n_i}$.
\end{enumerate}
Now, define the two global elements of the full symmetric group on $\Omega$. 

Define on each block $\Omega_i,$ $\pi_i=(w_i^1 w_i^2 w_i^3).$ Thus, \[\pi=\prod_i \pi_i\] acts on the first three points of each block $\Omega_i$. Since a 3-cycle is even, each $\pi_i\in A_{n_i}$ and $\pi\in G$.

Define another element $\tau$ as follows, on each block $\Omega_i$, consider the full $n_i$-cycle, \[\tau_i=(w_i^1 w_i^2\dots w_i^{n_i}).\] Thus,  \[\tau=\prod_i \tau_i\] acts on the first $n_i$-points of each block $\Omega_i$. Unlike, element $\pi$, $\tau$ is even only when $n_i$ is odd, therefore, $\tau$ need not belong to the group $G$. However, it always belong to \[G_1=\prod_i \operatorname{Sym}(\Omega_i),\] Now define \[\Gamma=\langle \tau, \pi, Q\rangle.\] Clearly, $\Gamma$ is a subgroup of $G_1$ but not necessarily a subgroup of $G$. As every generator preserves each $\Omega_i$ setwise, every $\Gamma$-orbit is contained in $\Omega_i$. 

$\Gamma$ is finitely generated because $Q$ is finitely generated see \cite[Section 2]{MR2031168}. We prove that the group $\Gamma$ with $Q=1$, $f(n)= 1$ for $n$ odd, and $f(n)=0$ for $n$ even, has infinite girth.
\begin{remark}
    We call the group $\Gamma_0=\langle \pi, \tau\rangle$, the B.H. Neumann group \cite{neumann1937some}. 
\end{remark}
\noindent Our next theorem provides an affirmative answer to Question \ref{compositionfactor}.

\begin{theorem}
\label{thm:neumann-composition-factors}
The B.H. Neumann group $\Gamma_0 = \langle \pi, \tau \rangle$ contains infinitely many distinct finite-index normal subgroups
\[
    N_i\trianglelefteq\Gamma_0
\]
such that:
\begin{enumerate}
    \item the quotients $\Gamma_0/N_i$ are pairwise non-isomorphic
    non-abelian finite simple groups;

    \item
    \[
        \gir(\Gamma_0)=\infty;
    \]

    \item $\Gamma_0$ contains no non-abelian free subgroup.
\end{enumerate}
In particular, the union of the sets of composition factors of the
finite quotients $\Gamma_0/N_i$ contains infinitely many pairwise
non-isomorphic non-abelian finite simple groups.
\end{theorem}

\begin{proof}
Let us recall that 
\[
    7\leq n_1<n_2<n_3<\cdots
\]
be a strictly increasing sequence of odd integers. For every $i$, let
\[
    \Omega_i
      =
    \{\omega_{i,1},\ldots,\omega_{i,n_i}\}
\]
be a set of cardinality $n_i$, and define permutations
\[
    \pi_i
      =
    (\omega_{i,1}\ \omega_{i,2}\ \omega_{i,3})
\]
and
\[
    \tau_i
      =
    (\omega_{i,1}\ \omega_{i,2}\ \cdots\ \omega_{i,n_i}).
\]
Since $n_i$ is odd, both $\pi_i$ and $\tau_i$ are even permutations,
so
\[
    \pi_i,\tau_i\in\operatorname{Alt}(\Omega_i).
\]

Set
\[
    \pi=(\pi_i)_{i\geq1},
    \qquad
    \tau=(\tau_i)_{i\geq1},
\]
and define the B.H. Neumann group as in \cite[Section 1]{MR2031168}
\[
    \Gamma_0
      :=
    \langle\pi,\tau\rangle
      \leq
    \prod_{i\geq1}\operatorname{Alt}(\Omega_i).
\]

By construction, $\Gamma_0$ is generated by the two elements
$\pi$ and $\tau$.

Let
\[
    \rho_i\colon
    \Gamma_0\longrightarrow\operatorname{Alt}(\Omega_i)
\]
be the restriction of the $i$-th coordinate projection. We claim that
$\rho_i$ is surjective. Its image contains $\pi_i$ and $\tau_i$, and
conjugating $\pi_i$ by powers of $\tau_i$ gives
\[
    \tau_i^j\pi_i\tau_i^{-j}
      =
    (\omega_{i,j+1}\ \omega_{i,j+2}\ \omega_{i,j+3}),
\]
where the indices are taken modulo $n_i$. The consecutive $3$-cycles
generate $\operatorname{Alt}(\Omega_i)$. Hence
\[
    \langle\pi_i,\tau_i\rangle
      =
    \operatorname{Alt}(\Omega_i),
\]
and therefore
\[
    \rho_i\colon
    \Gamma_0\twoheadrightarrow\operatorname{Alt}(\Omega_i)
\]
is an epimorphism.

Put
\[
    N_i:=\ker(\rho_i).
\]
Then $N_i$ is a finite-index normal subgroup of $\Gamma_0$, and the
first isomorphism theorem gives
\[
    \Gamma_0/N_i
      \cong
    \operatorname{Alt}(\Omega_i)
      \cong
    \operatorname{Alt}(n_i).
\]
Since $n_i\geq7$, every $\operatorname{Alt}(n_i)$ is a non-abelian
finite simple group. Moreover, since the $n_i$ are strictly
increasing, these groups are pairwise non-isomorphic.

It follows also that the subgroups $N_i$ are pairwise distinct.
Indeed, if $N_i=N_j$, then
\[
    \operatorname{Alt}(n_i)
      \cong
    \Gamma_0/N_i
      =
    \Gamma_0/N_j
      \cong
    \operatorname{Alt}(n_j),
\]
which would imply $n_i=n_j$.

Since the finite quotients $\operatorname{Alt}(n_i)$ have unbounded
orders, $\Gamma_0$ must be infinite. Furthermore, the group $\Gamma_0$ is elementary amenable; see
\cite[Corollary 1.2]{MR2031168}. Every subgroup of an amenable group is amenable, whereas a non-abelian free group is non-amenable. Therefore $\Gamma$ contains no non-abelian free subgroup.

Finally, the epimorphisms
\[
    \rho_i\colon
    \Gamma_0\twoheadrightarrow\operatorname{Alt}(n_i)
\]
exhibit infinitely many pairwise non-isomorphic non-abelian finite
simple quotients. Theorem \ref{thm:redundant-nielsen-lifting} implies that
\[
    \gir(\Gamma_0)=\infty.
\]

Each quotient $\Gamma_0/N_i\cong\operatorname{Alt}(n_i)$ has the
composition series
\[
    1
      \trianglelefteq
    \operatorname{Alt}(n_i),
\]
whose unique composition factor is $\operatorname{Alt}(n_i)$.
Consequently, the finite quotients $\Gamma_0/N_i$ collectively possess
infinitely many pairwise non-isomorphic non-abelian composition
factors.
\end{proof}

\subsection{Girth of finitery alternative group $\operatorname{Alt}_\mathrm{fin}(\mathbb{Z})\rtimes \mathbb{Z}$}

In this subsection, we prove that a class of finitely generated lawless groups which surjects onto $\operatorname{Alt}_\mathrm{fin}(\mathbb{Z})\rtimes \mathbb{Z}$ has infinite girth. We need the following lemmas \ref{Altfinite}, \ref{G=AltX}, and \ref{singleextension} before proving that the group $\operatorname{Alt}_\mathrm{fin}(\mathbb{Z})\rtimes \mathbb{Z}$ has infinite girth.

\begin{lemma}\label{Altfinite}
Let $X$ be a set, let $U\subseteq X$ with $|U|\ge 3$, and let $T\subseteq X$ be a three-element subset such that $U\cap T\neq\varnothing$. Let $\gamma_T\in\operatorname{Sym}(X)$ be a $3$-cycle with  $\supp(\gamma_T)=T$ (fixing $X\setminus T$ pointwise). Consider $\operatorname{Alt}_{\mathrm{fin}}(U)$ as a subgroup of $\operatorname{Sym}(X)$ by extending every permutation of $U$ to fix $X\setminus U$ pointwise. Then
$$\left\langle\operatorname{Alt}_{\mathrm{fin}}(U),\gamma_T\right\rangle
=\operatorname{Alt}_{\mathrm{fin}}(U\cup \supp(\gamma_T))
=\operatorname{Alt}_{\mathrm{fin}}(U\cup T).$$
Here, if $T=\{a,b,c\}$, then $\gamma_T$ is either $(a\,b\,c)$ or its inverse $(a\,c\,b)$.
\end{lemma}

In Lemma \ref{Altfinite}, when $U$ and $T$ are finite, the Lemma says \[\langle \operatorname{Alt}(U),\gamma_T\rangle=\operatorname{Alt}(U\cup T).\]
\begin{proof}
 By hypothesis, $|U\cap T|=1,2$ or $3$. We will use the lemma in case-by-case basis.

 Case 1: $|U\cap T|=3$: As $|T|=3$, each element of $T$ intersects with $U$ meaning $T\subseteq U$. This immediately proves the assertion as $U\cup T=U$.
\medskip

 Case 2: $|U\cap T|=2$: 
 \begin{itemize}
     \item $|U|\geq 4$
     
     \noindent Write \[T=\{u_1,u_2,t\}, \qquad u_1,u_2\in U, \qquad t\in T\backslash U.\] We can assume $\gamma_T=(u_1 u_2 t)$ (replacing $\gamma_T$ by it's inverse if necessary). Clearly, $\operatorname{Alt}_{\mathrm{fin}}(U)$ is 2-transitive on the set $U$ when $|U|\geq 4$. Therefore, there exists $\sigma\in \operatorname{Alt}_{\mathrm{fin}}(U)$ such that \[\sigma(\{u_1,u_2\})=\{r,s\}, \qquad \text{for any distinct}\hspace{0.1cm}\{r,s\}\in U\backslash T.\] Assume $\sigma=(u_1 r)(u_2 s)$ . As $\sigma$ fixes $t$, \[\sigma^{-1}\gamma_T \sigma=\sigma^{-1}(u_1 u_2 t)\sigma=(rst)\] Thus, $\langle \operatorname{Alt}_\mathrm{fin}(U), \gamma_T\rangle$ contains all 3-cycles $(rst)$ where $r\neq s$ and $r,s\in U\backslash T$ and $t\in T$. It already contains 3-cycles supported in $U$. Hence, \[\langle \operatorname{Alt}_\mathrm{fin}(U), \gamma_T\rangle=\operatorname{Alt}_\mathrm{fin}(U\cup \{t\})= \operatorname{Alt}_\mathrm{fin}(U\cup T) \]

     \item $|U|=3$ Write $T=\{u_1,u_2,t\}$ and $U=\{u_1,u_2,u_3\}$ Take $\gamma_T=(u_1u_2t)$ and $\sigma=(u_1u_2u_3)$. Then, \[\sigma^{-1}\gamma_T\sigma=(u_1 tu_3)\] and the argument follow as in the previous subcase $|U|\geq 4$.

     \noindent 
 \end{itemize}
\medskip
Case 3: $|U\cap T|=1$: Write \[T=\{u_1,s,t\}\qquad u_1\in U, \qquad s,t\in T\backslash U\] Like Case 2, assume $\gamma_T=(u_1st)$ (replacing $\gamma_T$ by $\gamma_T^{-1}$ if necessary). $\operatorname{Alt}_\mathrm{fin}(U)$ acts transitively on $U$. Indeed, for any $u,v\in U$, choose $\sigma_v=(uvw)$ for any $w\in U\backslash \{u_1,u_2\}$, hence $\sigma_v(u)=v$. Clearly, $\sigma_v\in \operatorname{Alt}_\mathrm{fin}(U)$ being a 3-cycle and the choose of $w$ is guaranteed by $|U|\geq 3. $ Therefore, by transitivity of $\operatorname{Alt}_\mathrm{fin}(U)$, there exists $\sigma_v\in \operatorname{Alt}_\mathrm{fin}(U)$ such that $\sigma_v(u_1)=v$, for $u_1$ as defined in $\gamma_T$ and any $v\in U$. Moreover, $\sigma_v$ fixes $s,t$ \[\sigma_v \gamma_T \sigma_v^{-1}=(stv)\implies \sigma_v(u_1st)\sigma_v^{-1}=(vst).\] Hence, all the 3-cycles of the form $(vst)\in \operatorname{Alt}_\mathrm{fin}(U)$ for all $v\in U$. Moreover, for distinct $x,v\in U$, \[(vst)(xst)^{-1}=(xvs)\] Hence, $(xvs)\in \operatorname{Alt}_\mathrm{fin}(U)$ for all $x\neq v$ and $x,v\in U$. Together with $\operatorname{Alt}_\mathrm{fin}(U)$, it generates 
\[\operatorname{Alt}_\mathrm{fin}(U\cup \{s\}).\] Let's come back to the original cycle $\gamma_T=(u_1st)$. As $u_1,s\in U\cup \{s\}$, apply Case 2 replacing $U$ by $U\cup \{s\}$ \[\operatorname{Alt}_\mathrm{fin}(U\cup \{s\}, \gamma_T)=\operatorname{Alt}_\mathrm{fin}{(U\cup \{s,t\})}=\operatorname{Alt}_\mathrm{fin}(U\cup T).\] This completes the proof of the lemma.
 
\end{proof}
\begin{lemma}\label{altfin(X)asproductof3cycles}
    Every permutation $\sigma\in \operatorname{Alt}_\mathrm{fin}(X)$ is a product of 3-cycles.
\end{lemma}
\begin{proof}
    By definition, for any arbitrary $\sigma\in \operatorname{Alt}_\mathrm{fin}(X)$, $\supp(\sigma)=n<\infty$. So, it suffices to assume that $\sigma\in \operatorname{Alt}(n)$. We can write $\sigma$ as a product of an even number of transpositions.
    \medskip
    
    \noindent Case 1: If there exist transposition pairs sharing a common point\[(ab)(bc)\] then we replace each transposition pair $(ab)(bc)$ in $\sigma$ by $(abc)$.
    \medskip

    \noindent Case 2: If the transposition pairs are disjoint \[(ab)(cd)\] then replace each disjoint transposition pair $(ab)(cd)$ in $\sigma$ by $(acb)(dac)$.
    \medskip

    Since, $\sigma$ is an arbitrary element in $\operatorname{Alt}_\mathrm{fin}(X)$, we are done.
\end{proof}

\begin{lemma}\label{G=AltX}
    Let $X$ be a set, and let $\mathcal{C}$ be a collection of 3-cycles in $\operatorname{Sym}_\mathrm{fin}(X)$. If the subgroup $G=\langle \mathcal{C}\rangle$ acts transitively on $X$, then \[G=\operatorname{Alt}_\mathrm{fin}(X).\]
\end{lemma}
\begin{proof}
    Clearly, as every element of $\mathcal{C}$ is an even permutation (every 3-cycle is a product of transpositions), it follows 
    \begin{equation}\label{inclusion}
        G\leq \operatorname{Alt}_\mathrm{fin}(X).
    \end{equation} We first show that, for every $G$-orbit $Y\subseteq X$ containing at least three elements, the group induced by $G$ on $Y$ is \[G|_Y=\operatorname{Alt}_\mathrm{fin}(Y).\] Suppose $Y$ be a $G$-orbit containing at least three elements. We show that $G$ contains every 3-cycle supported in $Y$. By Lemma \ref{altfin(X)asproductof3cycles}, it follows that \[\operatorname{Alt}_\mathrm{fin}(Y)\leq G|_Y.\] To prove the other direction, let $\gamma_0\in \mathcal{C}$ be an element supported in $Y$, say \[Y_0=\supp(\gamma_0).\]Then if $|Y_0|=3$,
    \[\langle\gamma_0\rangle=\operatorname{Alt}_\mathrm{fin}(Y_0)\leq G.\] Now, we prove the following elementary enlargement observation. If $U\subseteq Y, |U|\geq 3,$ and $\operatorname{Alt}_\mathrm{fin}(U)\leq G$, and suppose there exists $\gamma\in \mathcal{C}$ satisfying $\supp(\gamma)\cap U\neq \emptyset$, then \[\operatorname{Alt}_\mathrm{fin}(U \cup \supp(\gamma))\leq G.\] Indeed, by Lemma \ref{Altfinite}, \begin{equation}\label{equationenlargement}
        \langle\operatorname{Alt}_\mathrm{fin}(U),\gamma\rangle=\operatorname{Alt}_\mathrm{fin}(U\cup \supp(\gamma)).
    \end{equation}For any finite $F\subseteq Y$, there exist finitely many $\gamma_1.\gamma_2,\dots, \gamma_m\in \mathcal{C}$ such that \[F\subseteq (Y_0 \cup \bigcup_{i=1}^m \supp(\gamma_i))\] reordering if necessary, for all $1\leq j\leq m$, \[\supp(\gamma_j)\cap (Y_o \cup\bigcup_{i\neq j} \supp(\gamma_i))\neq \emptyset \quad 1\leq i\neq j\leq m. \] Applying equation \ref{equationenlargement} repeatedly starting with $Y_0$, we get for \[\operatorname{Alt}_\mathrm{fin}(F)\leq G|_Y.\] As $F$ was arbitrary, $G|_Y$ contains every 3-cycles on $Y$. Hence, \[\operatorname{Alt}_\mathrm{fin}(Y)\leq G|_Y.\] The reverse inclusion follows from Equation \ref{inclusion}. Therefore, \[G|_Y=\operatorname{Alt}_\mathrm{fin}(Y).\] Now, as $G$ acts transitively on $X$, there is just one $G$-orbit of $X$, hence \[G=\operatorname{Alt}_\mathrm{fin}(X).\] 
\end{proof}

\begin{remark}\label{AltfinZseminolaw}
    The group $\operatorname{Alt}_{\mathrm{fin}}(\mathbb Z)\rtimes \mathbb{Z}$ satisfies no non-trivial law follows from \cite[Proposition 2.4]{akhmedov2003girth}, where it is shown that $\operatorname{Sym}_\mathrm{fin}(\mathbb{Z})\rtimes \mathbb{Z}$ do not satisfy a non-trivial law. We see that $\operatorname{Alt}_\mathrm{fin}(\mathbb{Z})\rtimes \mathbb{Z}$ is an index two subgroup in $\operatorname{Sym}_\mathrm{fin}(\mathbb{Z})\rtimes \mathbb{Z}$ and hence can not satisfy a non-trivial law.
\end{remark}

We now prove the following elementary fact about finite permutations, which will be used to “fill the gaps” in the proof of Theorem \ref{AltZsemiZ}. In fact, this lemma shows that any partial map can be extended by adding more points until it becomes a single complete cycle covering the entire set.
\begin{lemma}\label{singleextension}
    (Extension lemma) Let $X$ be a finite, and let $p$ be an injective partial map on $X$. Suppose that $p$ has no non-trivial periodic points. Then, after adjoining finitely many new points if necessary, $p$ extends to a single finite cycle.
\end{lemma}
\begin{proof}
    Because $X$ is finite, we can arrange it's elements as finitely many disjoint sequences \[(a_{i,1},a_{i,2},\dots, a_{i,r_i}), \quad 1\leq i\leq s,\] such that \[p(a_{i,j})=a_{i,j+1}\] whenever the left-hand side is defined. The absence of periodic points ensures that none of these finite sequences closes into a cycle. 

    Define a permutation by retaining all the prescribed values of $p$ and setting \[a_{i,r_i}\mapsto a_{i+1,1} \; (1\leq i<s), \quad a_{s,r_s}\mapsto a_{1,1}.\]The resulting permutation is a single cycle that extends $p$. Additional points may be inserted into these cycles if any prescribed length and parity is required.
\end{proof}

\begin{theorem}\label{AltZsemiZ}
    Let \[H=\operatorname{Alt}_\mathrm{fin}(\mathbb{Z})\rtimes \langle t \rangle,\quad t(k)=k+1.\]Then $\gir(H)=\infty$. More precisely, for every $N\geq 1$, there exists an odd cycle \[\sigma_N\in \operatorname{Alt}_\mathrm{fin}(\mathbb{Z})\] such that $H=\langle t,\sigma_N\rangle$ and no nonempty freely reduced word of length at most $N$ in $\{t^{\pm}, \sigma_N^{\pm}\}$ represents identity.
\end{theorem}
\begin{proof}
    Fix $N\geq 1$, and let \[\mathbb{F}_2=\langle x,y \rangle\]be the free group on two generator. Consider the ball of radius $N$ in $\mathbb{F}_2$ \[B_N=\{g\in \mathbb{F}_2:|g|_{\{x,y\}}\leq N\}.\]We will construct a finite cycle $\sigma_N$ such that the action of $t$ and $\sigma_N$ on a finite subset of $\mathbb{Z}$ generates $H$ with the desired result as in the statement of the theorem.
\medskip

    \noindent \textit{Step 1: Acting $x$ as a shift on $x$-cosets in $\mathbb{F}_2$}.
Consider the action of the cyclic group $\langle x\rangle$ on the set $B_N$ by left multiplication. This action partitions the $B_N$ into a finite collection of disjoint orbits. Since $B_N$ is finite, each orbit $\mathcal{O}$ is a finite subset of a right coset $\langle x\rangle g$. More precisely, for any $g\in B_N$, the orbit containing $g$ is the set \[\mathcal{O}_g=\{x^kg\in\mathbb{F}_2: x^kg\in B_N\}.\] Because, $\mathbb{F}_2$ is a free group (which has no relations), there exist unique integers $a_g\leq 0$ and $b_g\geq 0$ such that \[\mathcal{O}_g=\{x^kg:a_g\leq k\leq b_g\}.\] In some sense, each orbit represents an "interval of integers" of the form $x^{a_{g}}g,x^{a_{g}+1}g,\dots,$ $x^{b_g}g$.

Now, let $\{\mathcal{O}_1,\mathcal{O}_2,\dots, \mathcal{O}_m\}$ represents distinct orbits of $B_N$. Define $\beta: B_N\to \mathbb{Z}$ as follows \[\beta(x^kg_i)=S_i+k,\quad a_{g_i}\leq k\leq b_{g_i}\] where the offsets $S_i$ are chosen to ensure that the intervals of integers $[S_i+a_{g_i},S_i+b_{g_i}]$ are pairwise disjoint. By construction, $\beta$ is injective and satisfies the property:
\begin{equation}\label{translationofbeta}
    \beta(xg)=\beta(g)+1
\end{equation}whenever both $g$ and $xg$ belongs to $B_N$. Moreover, we can choose the offsets $S_i$ to be sufficiently large, we may ensure that the image $\beta(B_N)\subseteq \{2,3,\dots,M\}$ for some $M\in \mathbb{N}$.
\medskip

\noindent \textit{Step 2: Encoding the $y$-action as a partial permutation.} We now define a partial map $p:\beta(B_N)\to \beta(B_N)$ that mimics the action of the generator $y$. For any $\beta(g)\in \beta(B_N),$ define: 
\begin{equation}\label{definitionofp}
    p(\beta(g))=\beta(yg)
\end{equation} provided that $yg\in B_N$. This map $p$ is well-defined injective partial permutation of $\mathbb{Z}$ because $\beta$ is injective. To ensure that $p$ can be extended to a single finite cycle, we must verify that it possesses no periodic points. Suppose \[p^r(\beta(g))=\beta(g)\] for some $r\geq 1$, with all the intermediate iterates defined, then \[\beta(y^rg)=\beta(g)\] injectivity of $\beta$ gives $y^rg=g$, and $y^r=1$, which is impossible in $\mathbb{F}_2$. Hence the map $p$ has no non-trivial periodic points. 

Choose an integer $D>M$ sufficiently large that $[0,D]$ contains as many unused points as will be needed. Extend $p$ by prescribing \begin{equation}
    D\mapsto 0\mapsto 1.
\end{equation} These points are disjoint from $\beta(B_M)\subseteq \{2,3,\dots, M\}$. The resulting partial permutation still has no non-trivial periodic points.

By Lemma \ref{singleextension}, $p$ extends to a single cycle $\sigma_N$, whose support is contained in $[0,D]$. Inserting a few points if necessary, we may assume that $\sigma_N$ is odd and $>N$. Thus, \[\sigma_N\in \operatorname{Alt}_\mathrm{fin}(\mathbb{Z}).\]Hence, by construction (equation \ref{definitionofp}), \begin{equation}\label{sigmaN}
\sigma_N(D)=0,\quad \sigma_N(0)=1, \quad \text{and}\quad\sigma_N(\beta(g))=\beta(yg),\end{equation} whenever $g,yg\in B_N$.
\medskip

\noindent\textit{Step 3: Absence of short relations.} 
We claim that \[w(t,\sigma_N)\neq 1\]for every non-trivial freely reduced word $w=w(x,y)$ of length atmost $N$. We use the usual convention that the products of the permutations are composed from right to left. Write \[w=s_1s_2\dots s_l,\quad s_i\in \{x^{\pm 1}, y^{\pm 1}\}, \quad 1\leq i\leq l,1\leq l\leq N.\]Each suffix $s_js_{j+1}\dots s_l$ is a reduced word of length atmost $N$, and therefore belongs to $B_N$. Equation \ref{translationofbeta} and \ref{sigmaN}, together with their inverse versions ($t(\beta(x^{-1}g))=\beta(xx^{-1}g)=\beta(g)\implies t^{-1}(\beta(g))=\beta(x^{-1}g)$ and similarly for $\sigma^{-1}$), give \begin{equation}
    t(\beta(g))=\beta(xg),\quad t^{-1}(\beta(g))=\beta(x^{-1}g)
\end{equation}
\begin{equation}
\sigma_N(\beta(g))=\beta(yg),\quad \sigma_N^{-1}(\beta(g))=\beta(y^{-1}g)
\end{equation}
whenever the elements $g,x^{\pm 1}g, y^{\pm 1}g\in B_N$. Thus, \[w(t,\sigma_N)(\beta(1))=\beta(w).\]
Since $w\neq 1$ and $\beta$ is injective, \[\beta(w)\neq \beta(1).\]Therefore, $w(t,\sigma_N)\neq 1.$ Hence, this proves that for every non-trival reduced words $w$ of length atmost $N$,  $w(t,\sigma_N)\neq 1.$
\medskip

\noindent \textit{Step 4: $t^D\sigma_Nt^{-D}\cap\sigma_N=\{D\}.$}
Suppose $A=\supp(\sigma_N)$. As $A\subseteq [0,D]$, we have $A+D\subseteq [D,2D]$. Therefore, \[A\cap (A+D)=\{D\}.\] Because \[\supp(t^D\sigma_N t^{-D})=t^D(\supp (\sigma_N))\subseteq[D,2D].\] Hence, we proved that $t^D\sigma_Nt^{-D}\cap \sigma_N=\{D\}.$
\medskip

\noindent \textit{Step 5: $[\sigma_N,t^D\sigma_Nt^{-D}]=(D,\;0,\;D+1)$.} Suppose $a$ and $b$ are two cycles whose support intersects exactly at $q$. Then \[[a,b]=aba^{-1}b^{-1}=(q,\;a(q),\;b(q)).\]  This follow from the fact that two cycles sharing exactly one point has a commutator that is 3-cycle. In our case, take \[a=\sigma_N,\quad b=t^D\sigma_N t^{-D}, q=D\] This prescribed values gives us, \[a(D)=\sigma_N(D)=0, \quad b(D)=t^D\sigma_N t^{-D}(0)=D+\sigma_N(0)=D+1.\] Hence, \begin{equation}\label{commutator}
    [\sigma_N,t^D\sigma_N t^{-D}]=(D\; 0\; D+1)
\end{equation}
\medskip

\noindent \textit{Step 6: $\{\sigma_N, t\}$ generates $H$ and  $\gir(H)=\infty$.} Conjugating equation \ref{commutator} by $t^n$ gives 
\begin{equation}\label{gammaDn}
    \gamma_{D,n}:=t^n[\sigma_N,t^D\sigma_N t^{-D}]t^{-n}=(D+n,\; n,\; D+n+1)\in \langle \sigma_N, t\rangle. 
\end{equation} Plugging in $D=1$ gives the 3-cycles set \[\mathcal{C}=\{\gamma_{1,n}=(n+1,\; n,\; n+2): n\in \mathbb{Z}\}.\] $\langle\mathcal{C}\rangle$ acts transitively on $\mathbb{Z}$, this is because the $\supp(\gamma_{1,n})=\{n, n+1, n+2\}$ connect every integer to every other integer in a finitely many steps. Hence, Lemma \ref{G=AltX} gives 
\begin{equation}\label{gammaonen}
    \langle \gamma_{1,n}:n\in \mathbb{Z} \rangle=\operatorname{Alt}_\mathrm{fin}(\mathbb{Z}).
\end{equation} Equations \ref{gammaonen} and \ref{gammaDn} gives \[\operatorname{Alt}_\mathrm{fin}(\mathbb{Z})\leq \langle t, \sigma_N\rangle\] Since, $t\in \langle t,\sigma_N \rangle$, we obtain \[\operatorname{Alt}_\mathrm{fin}(\mathbb{Z})\rtimes \langle t\rangle=\langle t,\alpha_N\rangle=H.\] Thus, $\{t,\sigma_N\}$ generates $H$ and no reduced word of length at most $N$ represents identity, which implies \[\gir(H,\{t,\sigma_N\})>N\]As our choose of $N$ is arbitrary, this implies $\gir(H)=\infty$.

\end{proof}
\begin{remark}\label{remarksurjection}
    It is not difficult to show that $\Gamma_0=\langle \pi, \tau\rangle$ surjects onto $\operatorname{Alt}_\mathrm{fin}(\mathbb{Z})\rtimes \langle t\rangle=\langle c,t \rangle$ via $\pi\mapsto c$ and $\tau\mapsto t$, where $c=(0\;1\;2)$. As $\gir(\operatorname{Alt}_\mathrm{fin}(\mathbb{Z})\rtimes \langle t\rangle)=\infty$ by Theorem \ref{AltZsemiZ} and that would give another proof of the fact that $\gir(\Gamma_0)=\infty$ using Akhmedov quotient principle \cite[Proposition 1.1]{akhmedov2005girth}.
\end{remark}

\section{The Construction of $B_f$ and Its Girth Properties}

\noindent An element $\gamma\in \Gamma$ is said to be \textit{detected} by a finite quotient $Q$ of $\Gamma$ if there is a homomorphism $\phi: \Gamma \to Q$ such that $\phi(\gamma)\neq e$. The \textit{residual growth function} $F_{X,\Gamma}(n)$ is defined to be the minimal natural number $N$ such that any word $w$ of length $\leq n$ with respect to a generating set $X$ is detected by some finite quotient $Q$ of size at most $N$.

\begin{theorem} \cite[Theorem 1]{BouRabeeSeward+2016+199+204}\label{Bou}
    For any function $f:\mathbb{N}\to \mathbb{N}$, there exists a residually finite group $B_f$ and a two-element generating set $X$ for $B_f$ such that $F_{B_{f},X}(n)\geq f(n)$ for all $n\geq 8$.
\end{theorem}

The proof of Theorem \ref{Bou} relies on the construction of a finitely generated infinite group that embeds a sequence of alternating groups. We adopt the notation and terminology of \cite{BouRabeeSeward+2016+199+204}. Recall that for odd $n \geq 5$, the alternating group $\Alt(n)$ is simple and generated by the pair $\{\alpha_n, \beta\}$, where $\alpha_n = (1, 2, \dots, n)$ and $\beta = (1, 2, 3)$. It is easily verified that the conjugate $\alpha_n^{m} \beta \alpha_n^{-m}$ is the $3$-cycle $(m+1, m+2, m+3)$, and the commutator $[\alpha_n^{m} \beta \alpha_n^{-m}, \beta]$ is trivial if and only if $m \not\equiv \pm 1, \pm 2 \pmod{n}$.

Let $f \colon \mathbb{N} \to \mathbb{N}$ be an increasing function such that $f(n) \geq n$. We define sequences $p, q, d \colon \mathbb{N} \to \mathbb{N}$ such that $p$ is strictly increasing and odd with $p(1)=1$; $q(k)>2$ is odd; $d(k)=p(k)q(k)+2$ is odd and satisfies $d(k) > 2f(4p(k+1)+4)$; and for all $i \leq k$, $p(k+1)q(i) \not\equiv \pm 1, \pm 2 \pmod{d(i)}$. 

Let \[G = \prod_{k \geq 1} \Alt(d(k)),\] and for each $k$, let $G_k\cong \Alt(d(k))$ denote the canonical subgroup supported in the $k$-th coordinate, with $\pi_k \colon G \twoheadrightarrow \Alt(d(k))$ the corresponding projection. We define the elements $s, t \in G$ by:
\[ s = (\alpha_{d(k)}^{q(k)})_{k\geq 1}, \quad t = (\beta, \beta, \dots). \]
The group $B_f$ is defined as the subgroup $\langle s, t \rangle \leq G$. Since $B_f$ embeds into a direct product of finite groups, it is residually finite.

We recall that non-abelian free groups are residually alternating of odd degree \cite{katz1969residual}.

\begin{proposition}\label{nolawinBf}
    The group $B_f$ does not satisfy a non-trivial law.
\end{proposition}
\begin{proof}
By \cite{MR2705763, Schleimer2000}, if $B_f$ satisfies a non-trivial law in $k$ variables, it must satisfy one in two variables. Suppose $B_f$ satisfies a non-trivial law $w(x,y) = 1$ for some $w \in \mathbb{F}_2$. There exists a surjective homomorphism $\phi: \mathbb{F}_2 \twoheadrightarrow B_f$, and thus a non-trivial word $w' \in \mathbb{F}_2$ such that $\phi(w') = w(x,y) = 1$. However, since $\mathbb{F}_2$ is residually $\Alt(2n+1)$ \cite{katz1969residual}, for $n \geq M$, there exists $\psi: \mathbb{F}_2 \to \Alt(n)$ such that $\psi(w') \neq 1$. This contradicts the fact that for sufficiently large $k$, the composition $\pi_k \circ \phi: \mathbb{F}_2 \to \Alt(d(k))$ maps $w'$ to $1$.
\end{proof}

We now invoke Theorem \ref{thm:redundant-nielsen-lifting} to show that BouRabee-Seward group $B_f$ has infinite girth. Moreover, we show that $B_f$ admits a proper finite-index subgroup with infinitely many pairwise non-isomorphic non-abelian finite simple sections, a condition weaker than the one mentioned in Question \ref{compositionfactor}.
\begin{theorem}[Proper finite simple sections]
\label{thm:proper-finite-simple-sections}
The BouRabee-Seward group $B_f$ contains a proper finite-index normal subgroup $H\trianglelefteq B_f$ satisfying:
\begin{enumerate}
    \item $\operatorname{girth}(B_f)=\infty$;
       \item There are finite subgroups $1\neq N_i\trianglelefteq K_i<H$ such that $K_i/N_i \cong \Alt(n_i)$ for pairwise non-isomorphic non-abelian finite simple groups $\Alt(n_i)$.
\end{enumerate}
Moreover, each $\Alt(n_i)$ occurs as a proper subgroup of $H$.
\end{theorem}

\begin{proof}
Let $\Gamma = B_f$ as defined above. Recall the definitions:

\[s = (s_1, s_2, s_3, \dots) \; \text{and}\; t = (\beta, \beta, \beta, \dots)\]
where $s_k = \alpha_{d(k)}^{q(k)}, \beta = (1, 2, 3),$ and $\alpha_{d(k)}=(1, 2, 3, \dots, d(k))$.

For every $k\geq1$, the coordinate subgroup $G_k \leq B_f$ by \cite[proof of Theorem~1]{BouRabeeSeward+2016+199+204}, so the projection $\pi_k\colon B_f\twoheadrightarrow\Alt(d(k))$ is an epimorphism. Since $d(k)$ is unbounded, we choose a subsequence $m_1 < m_2 < \cdots$ such that the groups $\Alt(d(m_i))$ are pairwise non-isomorphic.

The coordinate epimorphisms $\pi_{m_i}$ and Theorem \ref{thm:redundant-nielsen-lifting} imply that $\operatorname{girth}^{=}_k(B_f)=\infty$ for $k \geq d(B_f)+1$, and thus $\operatorname{girth}(B_f)=\infty$.

We now construct the required proper finite index subgroup $H$. Choose an index $m_1$ from the subsequence above and define \[H := \ker(\pi_{m_1}),\] where $\pi_{m_1} : B_f\twoheadrightarrow \operatorname{Alt}(d(m_1))$. Then $H$ is a proper finite-index normal subgroup of $B_f$. For all $j \geq 2$, the coordinate subgroups $G_{m_j}$ are supported outside the $m_1$-st coordinate, so $G_{m_j} \leq H$. Thus $H$ contains infinitely many pairwise non-isomorphic non-abelian finite simple subgroups.

To realize these as genuinely proper sections, divide the subsequence $\{m_j\}_{j \geq 2}$ into pairs. For $i \geq 1$, let $A_i = G_{m_{2i}}$ and $C_i = G_{m_{2i+1}}$. Since they are supported in distinct coordinates, they commute and $A_i \cap C_i = \{1\}$. Setting $K_i = A_i \times C_i$ and $N_i = C_i$, we have $1 \neq N_i \trianglelefteq K_i < H$ and $K_i/N_i \cong (A_i\times C_i)/C_i \cong A_i \cong \Alt(d(m_{2i}))$. Because $H$ is infinite as it contains $G_{m_j}$ of unbounded order and $K_i$ is finite in $H$, the subgroup $K_i$ is proper in $H$. Finally, each $A_i \cong\operatorname{Alt}(d(m_{2i}))$ is a proper finite subgroup of the infinite group $H$, completing the proof.
\end{proof}







\section{Marked finite quotient}\label{marked}
In this section, we introduce the notion of marked finite quotients. In this setting, the generating set in the quotient must be the image of a generating set coming from the parent group. Specifically, we require that:

\begin{align*}
(\ast) \qquad \forall N \ge 1, \exists\, S_N \subset \Gamma \text{ finite generating set, a finite group}\;Q_N,\;\text{and}\text{ a surjective}\\ \text{homomorphism } \pi_N \colon \Gamma \twoheadrightarrow Q_N
\text{ such that }
\operatorname{girth}(Q_N, \pi_N(S_N)) > N. 
\end{align*}

Under this marked interpretation, we prove that if the quotient generating set comes from the generating set of the group $\Gamma$, then we have an affirmative answer to Question \ref{pyberresidualfiniteness1}.  

\begin{theorem}\label{quoitientfinitenessgirth}
    Let $\Gamma$ be a finitely generated group. Suppose $\Gamma$ satisfies the condition $(\ast)$, then $\operatorname{girth}(\Gamma) = \infty$.
\end{theorem} 
\begin{proof}
   To the contrary, suppose $\gir(\Gamma)=M<\infty$. Then by definition of girth, for all finite generating set $S$ of $\Gamma$, we have \[\gir(\Gamma,S)\leq M.\] Taking $N>M$, by the hypothesis, there exists finite generating set $S_N$ of $\Gamma$ and surjective $\pi_N:\Gamma\to Q_N$ such that \[\gir(Q_N,\pi_N(S_N))>N,\]
 we claim that $\gir(Q_N,\pi_N(S_N))\leq \gir(\Gamma,S_N)$.

 Indeed, every relations in $\Gamma$ (with respect to $S_N$) remains a relation under $\pi_N$ in $Q_N$. More precisely, consider a relation in $\Gamma$ \[w(s_1,s_2,\dots,s_k)=1,\] applying $\pi_N,$ we have \[1=\pi_N(1)=\pi_N(w(s_1,s_2,\dots,s_k))=w(\pi_N(s_1)\pi_N(s_2)\dots\pi_N(s_k)),\] which is a relation in $Q_N$ with respect to the generating set $\pi_N(S_N)$. However,  finite quotient $Q_N$ could possibly have shorter relations which are not in the relations coming from $\Gamma$, therefore we must have \[N< \gir(Q_N,\pi_N(S_N))\leq \gir(\Gamma,S_N)\leq M\] which is a contradiction to the fact that $N>M$.
\end{proof}

\begin{remark}
    Let us note that the statement in Theorem \ref{quoitientfinitenessgirth} would be false if it were formulated as: “If, for every $N\in\mathbb{N}$, there exists a finite quotient $Q_N$ such that $\gir(Q_N)>N$, then $\gir(\Gamma)=\infty$.” Indeed, by the definition of girth, there exists a generating set $T_N$ of $Q_N$ with arbitrarily large girth, but there need not exist a generating set $S_N$ of $\Gamma$ such that $T_N=\pi_N(S_N)$. This is precisely what happens for $\mathbb{Z}^2$ in Theorem \ref{counterexample}.
\end{remark}

\begin{remark}
    We also note that Theorem \ref{quoitientfinitenessgirth} does not contradict the proof of \cite[Proposition 1.1]{akhmedov2005girth}, where it is shown that if $\gir(G)=m$, then $\gir(G/N)=Lm$ for some constant $L$ depending on $m$, for a normal subgroup $N\leq G$.
\end{remark}

The converse to the Theorem \ref{quoitientfinitenessgirth} holds for the class of residually finite groups. The following marked finite-quotient criterion is an elementary consequence of residual finiteness and the characterization of infinite girth through asymptotic freeness due to Arzhantseva \cite[Definition 3.1 and Theorem 3.3]{arzhantseva2014asymptotic}.

\begin{theorem} \label{markedfinitequotient}(Marked finite quotient criterion)
    Let $\Gamma$ be a finitely generated residually finite group. The following are equivalent:
    \begin{enumerate}
        \item $\gir(\Gamma)=\infty$.
        \item  $\Gamma$ satisfies the condition $(\ast)$.
    \end{enumerate}
\end{theorem}

\begin{proof}
    $(2) \implies (1)$ follows from Theorem \ref{quoitientfinitenessgirth}.

    We will prove $(1) \implies (2)$,  Suppose \[\gir(\Gamma)=\infty.\]
    Then, for any every $N\in \mathbb{N}$, there exists a finite generating set $S_N$ of $\Gamma$, we have \[\gir(\Gamma,S_N)>N.\] Let $\mathcal{B}_N$ be the collection of all nontrivial freely reduced words with respect to $\{S_N,S_N^{-1}\}$ of length at most $N$. Since \[\gir(\Gamma,S_N)>N,\] $w(S_N)\neq 1$ for every $w\in \mathcal{B}_N$. By residual finiteness of $\Gamma$, for each $w\in B_N$, there exists a quotient $\phi_w:\Gamma\to Q_w$ such that $\phi_w(w(S_N))\neq 1$. Now, consider the product map, \[\Phi:\Gamma\to \prod_{w\in B_N} Q_w.\] Then, no non-trivial freely reduced word in $w\in \mathcal{B}_N$ evaluates to identity on $\Phi(S_N)$ in the image subgroup $\Phi_N(S_N)$. Letting $Q_N=\prod_{w\in B_N}Q_w$ we get \[\gir(Q_N.\Phi(S_N))>N,\] implying (2).
\end{proof}

\begin{remark}
Theorem \ref{markedfinitequotient} shows that for the class of residually finite groups, it does not just give a sufficient condition but it gives a exact marked finite quotient reformulation (the images of some specific generating sets are remembered) of infinite girth.
\end{remark}

However, Theorem \ref{markedfinitequotient}, does not say that the infinite girth is determined by the unmarked set \[\mathcal{C}(\Gamma)=\{Q:Q \hspace{0.05in} \text{is a finite quotient of} \hspace{0.05in} \Gamma\}\] Therefore, the generating tuples and the quotient maps matter. This is important because we can not the transfer girth information to any arbitrary finite quotients, in context of the profinite rigidity studied by Bridson, McReynolds, Reid, and Spitler, \cite{bridson2020absolute}.

\subsection{Infinite girth for profinitely rigid groups}
In this section, we investigate girth questions of rigidity inspired by the work of Bridson, Conder, and Reid \cite{bridson2020absolute}. To set the context, recall that two finitely generated groups with isomorphic profinite completions share the same collection of finite quotients up to isomorphism. This leads to the study of absolute profinite rigidity: we say that a finitely generated residually finite group $\Gamma$ is absolutely profinitely rigid if the condition $\widehat{\Lambda} \cong \widehat{\Gamma}$ implies $\Lambda \cong \Gamma$ for any finitely generated residually finite group $\Lambda$. In this sense, profinite rigidity provides a framework for determining whether the totality of a group's finite quotients is enough to characterize its abstract isomorphism type.
\begin{proposition}\label{girthprofiniterigidity}
    Let $\Gamma$ be an absolutely profinitely rigid, finitely generated residually finite group. Suppose \[
    \gir(\Gamma)=\infty.
    \]If $\Lambda$ is any finitely generated residually finite group satisfying \[\widehat{\Lambda}\cong\widehat{\Gamma}\] then $\gir(\Lambda)=\infty.$ 
\end{proposition}
\begin{proof}
    Absolute profinite rigidity from \cite{bridson2020absolute} gives $\Lambda\cong \Gamma$, that is an isomorphims $\phi:\Lambda \to \Gamma$. A group isomorphism sends generators to generators and preserves all relations with respect to corresponding marked generating set, hence, $\gir(\Lambda)=\infty$ implies $\gir(\Gamma)=\infty.$ by Akhmedov quotient principle \cite[Proposition 1.1]{akhmedov2005girth}.
\end{proof}

The following corollary is an immediate consequence of the Proposition \ref{girthprofiniterigidity}.
\begin{corollary}
    Let $\Gamma$ be one of the non-elementary Kleinian group, Fuchian groups or hyperbolic groups which is known to be absolutely profinitely rigid \cite{bridson2020absolute}. If $\Lambda$ is finitely generated and residually finite with \[\widehat{\Lambda}\cong\widehat{\Gamma}\] then $\gir(\Lambda)=\infty.$
\end{corollary}
\begin{example}\label{hyperbolictrianglegroup}
    For the triangle group, \[\Delta(p,q,r)=\{x,y | x^p=y^q=(xy)^r=1\}\] if $\frac{1}{p}+\frac{1}{q}+\frac{1}{r}<1,$ then $\Delta(p,q,r)$ is non elementary hyperbolic and $\gir(\Delta(p,g,r))=\infty,$ by \cite[Theorem 2.6]{akhmedov2003girth}. As hyperbolic triangle groups are linear \cite{calegari2025fields} and finitely generated linear groups are residually finite, therefore, every hyperbolic triangle group are residually finite \cite{MR3420}. We can use Theorem \ref{markedfinitequotient}, for every $N>0$, there exist some finite generating set $S_N$ of $\Delta(p,q,r),$ finite quotient $Q_N,$ and $\pi_N:\Delta(p,q,r)\to Q_N$ such that $\gir(Q_N,\pi_N(S_N))>N$. However, note that $S_N$ can not be chosen to be the standard generating set as indeed in every finite quotient, \[\pi_N(x)^p=1, \pi_N(y)^q=1, (\pi_N(x)\pi_N(y))^r=1.\] $\gir(Q_N,\pi_N(\{x,y\}))\leq \min\{p,q,r\}.$
\end{example}
\begin{remark}\label{rem_girthk}
    Following \cite[Definition 4.7]{Akhmedov_2025}, for $k \geq d(\Gamma)$, where $d(\Gamma)$ denotes the minimum cardinality of a generating set for $\Gamma$, we define 
    \[ \operatorname{girth}_k(\Gamma) = \sup_{\substack{S \text{ generates } \Gamma \\ |S| \leq k}} \operatorname{girth}(\Gamma, S). \]
    Infinite girth, as defined in Section \ref{sec:introduction}, does not automatically specify the smallest $k$ for which $\operatorname{girth}_k(\Gamma) = \infty$. It is therefore natural to ask whether, for a hyperbolic triangle group $\Delta(p,q,r)$ as defined in Example \ref{hyperbolictrianglegroup}, we have
    \[ \operatorname{girth}_2(\Delta(p,q,r)) = \infty? \]
    This question is significantly more subtle than the problem of ordinary infinite girth, as generating pairs of Fuchsian and triangle groups are subject to strong Nielsen-equivalence constraints. Indeed, there is a substantial body of literature dedicated to classifying the generating systems of such groups.
\end{remark}


\subsection{Congruence realization and rigidity of marked girth for linear groups}

Our next result strengthens Theorem \ref{markedfinitequotient} by taking the finite quotient to be congruence quotients of the linear groups. The following definitions are added here for completeness, which can be obtained in \cite[Section 3.4]{bourbaki1998commutative}.

\begin{definition}
    A ring $A$ is called a Jacobson ring if every prime ideal of $A$ is the intersection of a family of maximal ideals. 
\end{definition}

\begin{definition}
    A ring $A$ is a finitely generated $\mathbb{Z}$-algebra if there exists elements $a_1,a_2,\dots, a_n\in A$ such that every element of $A$ can be writen as $f(a_1,a_2,\dots, a_n)$ for some $f\in \mathbb{Z}[x_1,x_2,\dots,x_n]$.
\end{definition}
\medskip
We prove the following lemma for completeness, which can be found in any textbook in commutative algebra. 
\begin{lemma}\label{A/mfield}
    Let $A$ be a finitely generated $\mathbb{Z}$-algebra which is an integral domain. Let $a_1,a_2,\dots, a_s\in A$ be non-zero elements. Then there exists a maximal ideal $\mathfrak m\subset A$ such that \[a_i\not\in \mathfrak m\]for every $i=1,2,\dots, s$. Moreover, $A/\mathfrak m$ is a finite field.
\end{lemma}
\begin{proof}
    Let $a=a_1a_2\dots a_s.$ Since $A$ is an integral domain and $a_i\neq 0$, we must have $a\neq 0$. It suffice to find a maximal ideal $\mathcal{M}$ such that $a\not \in \mathcal{M}.$ Indeed, if $a_i\in \mathcal{M}$ for some $1\leq i\leq s$, then by definition of ideal $a\in \mathcal{M}$. It follows from \cite[Theorem 3, Section 3.4]{bourbaki1998commutative} $A$ is a Jacobson ring (as $\mathbb{Z}$ is Jacobson). Because $A$ is a Jacobson ring, it follows from \cite[Proposition 3, section 3.4]{bourbaki1998commutative} that for any ideal $\mathcal{B}\in A$, \[\bigcap_{\mathfrak m\in \mathcal{M}, \mathfrak m \supset \mathcal{B}} \mathfrak m=\bigcap_{p\in\mathcal{P}, p\supset \mathcal{B}} p=0\] The last equality follows from the fact that $A$ is an integral domain. Since, $a\neq 0$, $a$ cannot lie in all maximal ideals of $A$. Hence, there exists an maximal ideal $\mathfrak m_1\in \mathcal{M}$ such that \[a\not\in \mathfrak m_1.\] Therefore, $a_i\not \in \mathfrak m_1$ for all $1\leq i\leq s$. $A/\mathfrak m$ is a finite field follows again from \cite[Corollary 1, Section 3.4]{bourbaki1998commutative}.
\end{proof}
We now prove the main result of this section.

\begin{theorem}(Congruence realization and rigidity of marked girth)
    Let $\Gamma\leq\operatorname{GL}_d(K)$ be a finitely generated linear group over a field $K$ of characteristic zero. Suppose that $\gir(\Gamma)=\infty$. Then,  $\Gamma$ satisfies the condition $(\ast)$. More precisely, for every $N\in\mathbb N$, there exist a finite generating set $S_N$ of $\Gamma$, a finitely generated integral domain $A\subset K$ with
\[
\Gamma\leq\operatorname{GL}_d(A),
\]
and a maximal ideal $\mathfrak m_N$ of $A$ such that the congruence quotient
\[
\psi_{\mathfrak m_N}(\Gamma)
\leq\operatorname{GL}_d(A/\mathfrak m_N)
\]
satisfies
\[
\gir\bigl(
\psi_{\mathfrak m_N}(\Gamma),
\psi_{\mathfrak m_N}(S_N)
\bigr)>N.
\]
\end{theorem}

\begin{proof}
\textit{Step 1: Fixing a finitely generated coefficient ring.}

As $\gir(\Gamma)=\infty$, for any $N>0$, one can choose a finite generating set $S_N=\{u_{N1},u_{N2},\dots, u_{Nk}\}$ of $\Gamma$ such that \[\gir(\Gamma,S_N)>N.\] Let \[T=\{u_{N1}, u_{N2}, \dots, u_{Nk}, u_{N1}^{-1}, u_{N2}^{-1}, \dots, u_{Nk}^{-1}\}.\]   For each $t\in T$, which is a $d\times d$ with entries in $K$. Write down all the matrix entries for each $t\in T$, that is, \[t_1,t_2,\dots, t_r\] where $r\leq 2k.d^2$. Define $A$ to be the subring of $K$ generated by $\mathbb{Z}$ and $t_1, t_2, \dots, t_r$ that is 
\begin{equation}
    A=\mathbb{Z}[t_{1}, t_2, \dots, t_{r}]=\{f(t_{1}, t_{2}\dots, t_{r}): f\in \mathbb{Z}[x_{1}, x_2, \dots, x_{r}]\}\subset K
\end{equation}
which can be seen as the image of the evaluation map \[e_{\{t_{1},\dots,t_{r}\}}:\mathbb{Z}[x_{1}\dots,x_{r}]\to K\] defined as $e_{\{t_{1},\dots,t_{r}\}}(f)=f(t_{1},\dots,t_{r})$. Clearly, $A$ is a finitely generated $\mathbb{Z}$-algebra. Hence, $A$ is the coeffient ring we will be working with.
\medskip

$A$ is an integral domain follows from a basic fact from the ring theory, "every unital subring of an integral domain is an integral domain". Hence $A$ is an integral domain. We may now assume that $\Gamma\leq GL_d(A)$, as every element of $\Gamma$ has entries in $A$, so $\Gamma\leq GL_d(A)$.
\medskip

\textit{Step 2: Choosing a maximal ideal to obtain the finite residue field.}

Let \[\mathcal{W}_N=\{1\neq w\in F_k:|w|\leq N, w \hspace{0.1cm}\text{is freely reduced}\}.\]

As there are finitely many non-trivial freely reduced words $w$ of length at most $N$, the set $\mathcal{W}_N$ is finite. For $w\in \mathcal{W}_N$, we have \[w(S_N)\neq I_d.\]  Now, $w(S_N)-I_d$ has at least one non-zero entry. Choose one such entry and denote it by \[a_w\in A.\] Thus, $a_w\neq 0$. Now consider the product, \[a_N=\prod_{w\in \mathcal{W}_N}a_w\] As $A$ is an integral domain, $a_w\neq 0$ for all $w\in \mathcal{W}_N$, which implies $a_N\neq 0$. By Lemma \ref{A/mfield}, we have there exists a maximal ideal $\mathfrak m_N \subset A$ such that $a_N\not\in \mathfrak m_N$, implying each $a_w\not\in \mathfrak m_N$, for all $w\in \mathcal{W}_N$ (indeed, if $a_w\in \mathfrak m_N$ for some $w$, then $a_N\in \mathfrak m_N$, which is impossible).
\medskip

Again, by Lemma \ref{A/mfield}, $A/\mathfrak m_N$ is a finite field. So we have the following reduction mod $m_N$ homomorphism as $\Gamma\leq GL_d(A)$, \[\psi_{\mathfrak m_N}:\Gamma\to GL_d(A/\mathfrak m_N)\] Set $Q_N:=\psi_{\mathfrak m_N}(\Gamma)$ is a finite subgroup of $GL_d(A/\mathfrak m_N)$. Moreover $S_N'=\psi_{\mathfrak m_N}(S_N)=\{\psi(u_{N1}), \psi(u_{N2}),\dots, \psi(u_{Nk})\}$ is a finite generating set of $Q_N$ (because $S_N$ generates $\Gamma$). For $w\in \mathcal{W}_N$, suppose \[\psi_{\mathfrak m_N}(w(S_N))=I_d \implies \psi_{\mathfrak m_N}(w(S_N)-I_d)=0.\] This means any chosen entry $a_w$ of $w(S_N)$ would be $0$ modulo $\mathfrak m_N$, which further means $a_w\in \mathfrak m_N$, which is impossible by the choice of $\mathfrak m_N$. Hence, for any non-trivial $w\in \mathcal{W}_N$, we have \[w(S_N')=w(\psi_{\mathfrak m_N}(S_N))=\psi_{\mathfrak m_N}(w(S_N))\neq 1\] in the finite quotient $Q_N$, therefore \[\gir(Q_N,S_N')=\gir(\psi_{\mathfrak m_N}(\Gamma), \psi_{\mathfrak m_N}(S_N))>N.\]
    
\end{proof}


Biggs gave an explicit construction of finite permutation groups whose Cayley graphs have valence $C$ and girth
tending to infinity as the radius $R$ of the tree tends to infinity \cite{biggs1998constructions}. We refer to the notion of asymptotic approximation introduced by Arzhantseva \cite[Definition 3.1]{arzhantseva2014asymptotic}. The following proposition reproduces \cite[Theorem 3.3]{arzhantseva2014asymptotic}, with an elementary proof included for completeness. We thank Goulnara Arzhantseva for bringing this result to our attention. 

\begin{proposition} (Arzhantseva \cite[Theorem 3.3]{arzhantseva2014asymptotic})
\label{prop:infinite-girth-asymptotically-rf}
Let $\Gamma$ be a finitely generated group. If
\[
    \operatorname{girth}(\Gamma)=\infty,
\]
then $\Gamma$ is asymptotically free and, consequently,
asymptotically residually finite.
\end{proposition}

\begin{proof}
For every $n\geq1$, choose a finite generating set
\[
    S_n=\{s_{n,1},\ldots,s_{n,r_n}\}
\]
of $\Gamma$ such that
\[
    \operatorname{girth}(\Gamma,S_n)>3n.
\]
Let
\[
    F_n
      :=
    F(x_{n,1},\ldots,x_{n,r_n})
\]
be the free group freely generated by a set in bijection with $S_n$.

Denote by
\[
    B_{S_n}(n)
      =
    \{g\in\Gamma:|g|_{S_n}\leq n\}
\]
the ball of radius $n$ in $\Gamma$ with respect to $S_n$.

For every $g\in B_{S_n}(n)$, choose a word
\[
    u_g=u_g(x_{n,1},\ldots,x_{n,r_n})
\]
of length at most $n$ such that
\[
    g=u_g(s_{n,1},\ldots,s_{n,r_n}),
\]
and define
\[
    \iota_n\colon B_{S_n}(n)\longrightarrow F_n,
    \qquad
    \iota_n(g):=u_g.
\]

We first verify that $\iota_n$ is well-defined. Suppose that $u$ and
$v$ are two words of length at most $n$ representing the same element
of $\Gamma$. Then
\[
    uv^{-1}=1
\]
in $\Gamma$, and
\[
    |uv^{-1}|\leq2n.
\]
Since
\[
    \operatorname{girth}(\Gamma,S_n)>3n,
\]
the freely reduced form of $uv^{-1}$ must be trivial. Hence $u=v$ in
the free group $F_n$.

The same argument shows that $\iota_n$ is injective. Indeed, if
\[
    \iota_n(g)=\iota_n(h),
\]
then the corresponding words are equal in $F_n$, and therefore they
have the same evaluation in $\Gamma$, giving $g=h$.

We next verify that $\iota_n$ preserves multiplication whenever the
multiplication is visible in the ball. Let
\[
    g,h,gh\in B_{S_n}(n).
\]
Choose words $u_g,u_h,u_{gh}$ of length at most $n$ representing
these elements. In $\Gamma$, one has
\[
    u_gu_hu_{gh}^{-1}=1.
\]
The length of this word is at most $3n$. Since the marked girth is
greater than $3n$, it follows that
\[
    u_gu_hu_{gh}^{-1}=1
\]
already in the free group $F_n$. Thus
\[
    \iota_n(g)\iota_n(h)=\iota_n(gh).
\]

Therefore
\[
    \iota_n\colon B_{S_n}(n)\longrightarrow F_n
\]
is an injective partial homomorphism. It follows that the sequence of
marked groups $(\Gamma,S_n)$ is asymptotically approximated by free
groups.

Since free groups are residually finite, $\Gamma$ is asymptotically
residually finite.
\end{proof}

\begin{corollary}
\label{cor:groups-asymptotically-rf}
Every group shown in this paper to have infinite girth is
asymptotically free and, in particular, asymptotically residually
finite.
\end{corollary}

\begin{proof}
This follows immediately from
Proposition~\ref{prop:infinite-girth-asymptotically-rf}.
\end{proof}

\section{Open Questions and remarks}
In following result of Pyber \cite[Theorem C]{MR2031168}, provides a particularly strong failure of absolute
profinite rigidity. More precisely, Pyber constructs a continuum
$\{\Gamma_\alpha\}_{\alpha\in I}$ of pairwise nonisomorphic
four-generated residually finite groups such that
\[
    \widehat{\Gamma_\alpha}\cong\widehat{\Gamma_\beta}
    \qquad(\alpha,\beta\in I).
\]
Consequently, no group $\Gamma_\alpha$ in this family is absolutely
profinitely rigid. 

\begin{question}\label{girthasprofiniteinvartiant}
   Does there exist a pair of non-isomorphic groups $\Gamma_\alpha, \Gamma_\beta$ in the family constructed by Pyber \cite[Theorem C]{MR2031168} such that $\gir(\Gamma_\alpha) = \infty$ while $\gir(\Gamma_\beta) < \infty$?
\end{question}

It is shown by Akhmedov \cite{MR2705763} that girth type is not a quasi-isometry invariant. Our Question \ref{girthasprofiniteinvartiant} is asking whether girth is a profinite invariant.

\begin{remark}
\cite[Theorem C]{MR2031168} does not by itself imply that girth is not a
profinite invariant. To obtain such a conclusion, one would have to
exhibit two groups $\Gamma_\alpha$ and $\Gamma_\beta$ in Pyber's
family satisfying
\[
    \gir(\Gamma_\alpha)=\infty
    \quad\text{and}\quad
    \gir(\Gamma_\beta)<\infty.
\]
The existence of continuum many nonisomorphic groups with the same
profinite completion only establishes the failure of absolute
profinite rigidity; it does not force variation of girth within the
family.
\end{remark}

This following proposition shows satisfying group law as an obstruction to infinite girth for the class of residually finite groups with isomorphic profinite completion.

\begin{proposition}\label{prop:laws-profinite}
Let $G$ and $K$ be residually finite groups. If their profinite completions are isomorphic, $\widehat G \cong \widehat K$, then $G$ and $K$ satisfy exactly the same group laws.
\end{proposition}

\begin{proof}
Suppose that $w(x_1,\ldots,x_r)$ is a group law in $G$. By definition, this means that for every $r$-tuple of elements $(g_1, \ldots, g_r) \in G^r$, the word evaluates to the identity: $w(g_1, \ldots, g_r) = 1_G$.

We consider the word map $w \colon \widehat{G}^r \to \widehat{G}$, defined by extending the group operations to the profinite completion. Since the group operations in $\widehat{G}$ are continuous by construction, the map $w$ is a continuous function from the product space $\widehat{G}^r$ to $\widehat{G}$. 

Because $G$ is residually finite, it embeds densely into its profinite completion $\widehat{G}$. Consequently, the set $G^r$ is dense in $\widehat{G}^r$. Since $w$ is trivial on the dense subset $G^r$ (as $w$ is a law in $G$), and since the identity $\{1_{\widehat{G}}\}$ is a closed set in the profinite topology, the continuity of $w$ implies that $w$ must be trivial on the entire domain $\widehat{G}^r$. Thus, $w$ is a group law for $\widehat{G}$.

Now, let $\Phi \colon \widehat{G} \to \widehat{K}$ be the given isomorphism of profinite completions. Since $\Phi$ preserves the group structure, it maps the identity of $\widehat{G}$ to the identity of $\widehat{K}$. Therefore, the word $w$ is also a law in $\widehat{K}$. Restricting this result to the dense subgroup $K \leq \widehat{K}$, we conclude that $w(k_1, \ldots, k_r) = 1_K$ for all $k_i \in K$. Hence, $w$ is a law in $K$.

The reverse implication follows by symmetry, interchanging the roles of $G$ and $K$.
\end{proof}

\begin{question}\label{intermediategrouphasinfinitegirth}
 Does every finitely generated group of intermediate growth have infinite girth?
\end{question}
\begin{definition}
    A $k$-generated group $G$ is said to satisfy an \emph{almost identity} if for all $k\in \mathbb{N}$, there exists a non-trivial freely reduced word $w(x_1,x_2,\dots,x_k)\in \mathbb{F}_k$ such that if \[w(g_1, g_2, \dots, g_k)=1, \quad g_1,g_2,\dots,g_k\in G,\] then $\{g_1,g_2,\dots, g_k\}$ generates the group $G$.
\end{definition}
\begin{remark}
While the first Grigorchuk group $\mathfrak{G}$ provides a positive example (by \cite[Corollary 6.13] {bartholdi2015ordering}) to Question \ref{intermediategrouphasinfinitegirth}, it remains an open problem whether intermediate growth generally obstructs the existence of almost-identities, \cite[Question 8.6]{bartholdi2015ordering}. If there exists a group of intermediate growth with finite girth, it would necessarily satisfy a non-trivial almost-identity, providing a positive answer to the question posed by Bartholdi and Erschler \cite{bartholdi2015ordering}.
\end{remark}

\subsection{Infinite Girth in Groups of Oscillating Intermediate Growth}

In this subsection, we apply the Akhmedov quotient principle \cite[Proposition 1.1]{akhmedov2005girth} to a class of groups with highly irregular growth. We specifically consider the groups of oscillating intermediate growth constructed by Kassabov and Pak \cite{MR3034295}. These groups are of particular interest as they provide examples of finitely generated residually finite groups whose growth rates fluctuate between different intermediate scales.

\begin{proposition}\label{KPinfinitegirth}
Let $\Gamma$ be a group of oscillating intermediate growth obtained from the Kassabov--Pak construction based on the first Grigorchuk group $\mathfrak{G} = G_{(012)^\infty}$. Then $\Gamma$ is residually finite and satisfies
\[
    \gir(\Gamma) = \infty.
\]
Consequently, the Kassabov--Pak construction yields an uncountable family of residually finite groups of oscillating intermediate growth and infinite girth.
\end{proposition}

\begin{proof}
Recall that a group $\Gamma$ in the Kassabov--Pak construction is defined as the diagonal product
\[
    \Gamma = \bigotimes_{i \geq 1} G_i,
\]
where each $G_i$ is a finite group. By construction, $\Gamma$ embeds into the direct product $\prod_{i \geq 1} G_i$. Since the product of finite groups is residually finite and any subgroup of a residually finite group is itself residually finite, it follows immediately that $\Gamma$ is residually finite.

The fundamental property of this construction is that the sequence of marked groups $(G_i)$ converges to the base group $G_\omega$ in the space of marked groups. In the case where $G_\omega$ is chosen to be the first Grigorchuk group $\mathfrak{G}$, this convergence provides a canonical surjective homomorphism
\begin{equation}\label{mapsontogrigorchuk}
    \Pi \colon \Gamma \twoheadrightarrow \mathfrak{G}.
\end{equation}
A result of Bartholdi and Erschler \cite[Corollary 6.12]{bartholdi2015ordering} establishes that $\mathfrak{G}$ preforms the free group $F_3$; specifically, this implies that $\gir(\mathfrak{G}) = \infty$. By the Akhmedov quotient principle \cite[Proposition 1.1]{akhmedov2005girth}, $\gir(\Gamma)=\infty$ since $\Gamma$ admits a surjection onto a group of infinite girth.
\end{proof}

\begin{remark}
Proposition \ref{KPinfinitegirth} also provide an affirmative answer to our Question \ref{intermediategrouphasinfinitegirth} demonstrating that infinite girth is compatible with intermediate growth, even in the case of oscillating growth. Furthermore, since $\mathfrak{G}$ satisfies no non-trivial law \cite[Corollary 6.12]{bartholdi2015ordering}, \cite[Lemma 2.4]{bartholdi2015ordering}. Hence by (\ref{mapsontogrigorchuk}), $\Gamma$ is also law-free. Thus, this gives us an uncountable family of groups that simultaneously satisfy the following properties:
\begin{enumerate}
    \item $\Gamma$ is finitely generated and residually finite;
    \item $\Gamma$ possesses oscillating intermediate growth;
    \item $\Gamma$ satisfies no non-trivial group law;
    \item $\Gamma$ contains no non-abelian free subgroups;
    \item $\gir(\Gamma) = \infty$.
\end{enumerate}
\end{remark}

\subsection{Lawlessness, almost-identities, and infinite girth}
\label{subsec:lawlessness-almost-identities}

The absence of a non-trivial group law is a necessary, but not
sufficient, condition for infinite girth. Recent work of Bradford
\cite{MR4685112} provides a quantitative refinement of
lawlessness that helps clarify this distinction.

Let $\Gamma$ be a finitely generated group with finite generating set
$S$. For a non-trivial word $w\in F_k$ which is not a law in $\Gamma$,
Bradford defines the complexity of $w$ to be
\[
 \chi_\Gamma^S(w)
 =
 \min\left\{
   \sum_{i=1}^{k}|g_i|_S :
   w(g_1,\ldots,g_k)\neq 1
 \right\}.
\]
The lawlessness growth function $\mathcal A_\Gamma^S(n)$ is obtained by
taking the maximum of $\chi_\Gamma^S(w)$ over all non-trivial words of
length at most $n$.

Bradford proves that if $\Gamma$ is lawless, then
$\mathcal A_\Gamma^S$ is bounded if and only if $\Gamma$ contains a
non-abelian free subgroup. He also constructs finitely generated
elementary amenable lawless groups whose lawlessness growth is
arbitrarily slow, as well as lawless torsion groups with at least
linear lawlessness growth.

There is, however, an essential distinction between lawlessness and
infinite girth. If $\Gamma$ is lawless, then for every $R\geq 1$ there
exist $g_1,\ldots,g_k\in\Gamma$ such that
\[
  w(g_1,\ldots,g_k)\neq 1
\]
for every non-trivial word $w\in F_k$ of length at most $R$. The tuple
$(g_1,\ldots,g_k)$ need not generate $\Gamma$. By contrast, infinite
girth requires generating tuples with this property.

\begin{definition}
Let $\Gamma$ be a $k$-generated group. A non-trivial word
$w(x_1,\ldots,x_k)\in F_k$ is called a \emph{$k$-almost-identity} for
$\Gamma$ if
\[
  w(g_1,\ldots,g_k)=1
\]
for every generating $k$-tuple $(g_1,\ldots,g_k)$ of $\Gamma$.
\end{definition}

\noindent Let use recall the definition of $\gir^{=}_k(G)$ defined in section \ref{Nielsencomponent}, that is for $k\geq d(G)$
\[
  \gir^{=}_k(G)
  :=
  \sup\bigl\{
       \gir(G,\mathbf{g}) :
       \mathbf{g}\in G^k,\ 
       \langle g_1,\ldots,g_k\rangle=G
     \bigr\}.
\]
Here we are considering all generating sets of size exactly $k$, which is different from $\gir_k(\Gamma)$ defined in Remark \ref{rem_girthk}.

\begin{remark}
    Let us note that calling a group, with $H$ equipped with the basic choice of the metric, `$G$ preforms $H$’ (in the sense of Bartholdi and Erschler \cite[Corollary~6.7]{bartholdi2015ordering}) is same as saying ‘$G$ is asymptotically $H$’, in the sense of Arzhantseva \cite[Definition 3.1]{arzhantseva2014asymptotic}.
\end{remark}
\begin{proposition}\label{prop:almost-identity-girth}
Let $\Gamma$ be a $k$-generated group. Then the following are
equivalent:
\begin{enumerate}
    \item $\gir^{=}_k(\Gamma)=\infty$;
    \item $\Gamma$ satisfies no non-trivial $k$-almost-identity;
    \item $\Gamma$ preforms the free group $F_k$ or equivalently, $\Gamma$ is asymptotically $\{F_k\}$.
\end{enumerate}
\end{proposition}

\begin{proof}
The equivalence of (2) and (3) is due to Olshanskii and Sapir; see also
Bartholdi and Erschler \cite[Corollary~6.7]{bartholdi2015ordering}.

We give the direct connection with girth. For $(1)\implies (2)$, suppose first that
$\gir^{=}_k(\Gamma)=\infty$. If
$w\in F_k\setminus\{1\}$ were a non-trivial $k$-almost-identity, then every
generating $k$-tuple of $\Gamma$ would satisfy the relation $w=1$.
Consequently,
\[
  \gir^{=}_k(\Gamma)\leq |w|,
\]
which is a contradiction.

Conversely, for $(2) \implies (1)$, suppose that $\Gamma$ satisfies no non-trivial
$k$-almost-identity. Fix $R\geq 1$, and enumerate the non-trivial
freely reduced words in $F_k$ of length at most $R$ as
\[
  w_1,\ldots,w_m.
\]
There exists a non-trivial word $W\in F_k$ whose vanishing set in any
group contains the union of the vanishing sets of
$w_1,\ldots,w_m$. Since $W$ is not a $k$-almost-identity, there is a
generating $k$-tuple $(g_1,\ldots,g_k)$ of $\Gamma$ such that
\[
  W(g_1,\ldots,g_k)\neq 1.
\]
In particular
\[
  w_i(g_1,\ldots,g_k)\neq 1
  \qquad (1\leq i\leq m).
\]
Thus, the corresponding Cayley graph has girth greater than $R$.
Since $R$ was arbitrary,
\[
  \gir^{=}_k(\Gamma)=\infty.
\]
\end{proof}

\begin{question}
Do the elementary amenable lawless groups constructed by Bradford
\cite[Theorem~1.2]{MR4685112} have infinite girth? More
generally, which quantitative conditions on the lawlessness growth
function $\mathcal A_\Gamma$ imply the absence of almost-identities?
\end{question}



\section{Acknowledgements}
We are deeply grateful to László Pyber for raising several fundamental questions underlying this work and for his invaluable guidance. We thank Miklós Abért for discussions on the Nielsen-component lifting approach, which led to the resolution of Question 1.1 in the broadest sense and the formulation of Theorem C. We also thank Endre Szabó for suggesting the use of a large generating element $\sigma_N$ for each $N$, which was crucial to the proof of Theorem 3.8, and Hongyi Huang for raising the question of extending Theorem 2.11 to quasi-simple groups, leading to Theorem 2.19. We thank Goulnara Arzhantseva for her interest in this manuscript and bringing to our attention her work on asymptotic residual finiteness. OpenAI ChatGPT (model Sol 5.6) was used for linguistic polishing and grammatical corrections. All main mathematical ideas are due to the author, who takes full responsibility for the mathematical content.

\bibliography{name}
\bibliographystyle{amsalpha}
\end{document}